\documentclass[12pt]{article}
\usepackage{a4wide}
\usepackage{amsmath,amssymb,amsthm}
\usepackage[mathscr]{eucal}
\usepackage{tikz}
\usetikzlibrary{positioning,automata}
\usetikzlibrary{arrows,calc}
\usetikzlibrary{decorations.markings,plotmarks}
\usepackage{graphics}
\usepackage{color}
\usepackage{leftidx}

\definecolor{red}{rgb}{1,0.00,0.00}

\author{Ievgen Bondarenko and Bohdan Kivva}

\title{\textbf{Automaton groups and complete square complexes}}

\newcommand{\VH}{\mathcal{VH}}
\newcommand{\dual}{\partial}
\newcommand{\inv}{\textit{i}}

\newcommand{\St}{St}

\newtheorem{theorem}{Theorem}
\newtheorem{proposition}[theorem]{Proposition}
\newtheorem{corollary}[theorem]{Corollary}
\newtheorem{lemma}{Lemma}
\theoremstyle{definition}
\newtheorem{definition}{Definition}
\newtheorem{example}{Example}
\newtheorem{problem}{Problem}
\newtheorem{question}{Question}
\newtheorem{conjecture}{Conjecture}
\newtheorem{remark}{Remark}

\begin{document}

\newcommand{\SqComplex}{\Delta}

\newcommand{\WangTile}[4]{%
\begin{tikzpicture}[scale=1.3,baseline=-21pt,line width=0.5pt,decoration={
    markings,
    mark=at position 0.5 with {\arrow[xshift=3.333pt]{triangle 60}}},
    ]
 \filldraw[black] (0,0) circle (1pt) (0,-1) circle (1pt) (1,0) circle (1pt) (1,-1) circle (1pt);

 \draw[postaction={decorate}] (0,-1) -- node[left,xshift=-1.5pt,yshift=0.7pt] {#1} (0,0);
 \draw[postaction={decorate}] (0,0)  -- node[above,yshift=1.5pt] {#2} (1,0);
 \draw[postaction={decorate}] (0,-1) -- node[below,yshift=-1.5pt] {#3} (1,-1);
 \draw[postaction={decorate}] (1,-1) -- node[right,xshift=1.5pt,yshift=0.7pt] {#4} (1,0);

 \draw (0.2,-0.4) -- (0.4,-0.2);
 \draw (0.2,-0.6) -- (0.6,-0.2);
 \draw (0.2,-0.8) -- (0.8,-0.2);
 \draw (0.4,-0.8) -- (0.8,-0.4);
 \draw (0.6,-0.8) -- (0.8,-0.6);

\end{tikzpicture}
}

\newcommand{\ComplexFourNRF}{%
\begin{tikzpicture}[scale=1.5,baseline=-6mm,line width=0.5pt,decoration={
    markings,
    mark=at position 0.5 with {\arrow[xshift=3.333pt]{triangle 60}}},
    ]
 \filldraw[black] (0,0) circle (1pt) (0,-1) circle (1pt) (1,0) circle (1pt) (1,-1) circle (1pt);

 \draw[postaction={decorate}] (0,-1) -- node[left,xshift=-1.5pt,yshift=0.7pt] {a} (0,0);
 \draw[postaction={decorate}] (0,0) -- node[above,yshift=1.5pt] {x} (1,0);
 \draw[postaction={decorate}] (0,-1) -- node[below,yshift=-1.5pt] {y} (1,-1);
 \draw[postaction={decorate}] (1,-1) -- node[right,xshift=1.5pt,yshift=0.7pt] {b} (1,0);

 \draw (0.2,-0.4) -- (0.4,-0.2); \draw (0.2,-0.6) -- (0.6,-0.2); \draw (0.2,-0.8) -- (0.8,-0.2); \draw
(0.4,-0.8) -- (0.8,-0.4); \draw (0.6,-0.8) -- (0.8,-0.6);
\end{tikzpicture}\qquad
\begin{tikzpicture}[scale=1.5,baseline=-6mm,line width=0.5pt,decoration={
    markings,
    mark=at position 0.5 with {\arrow[xshift=3.333pt]{triangle 60}}},
    ]
 \filldraw[black] (0,0) circle (1pt) (0,-1) circle (1pt) (1,0) circle (1pt) (1,-1) circle (1pt);

 \draw[postaction={decorate}] (0,-1) -- node[left,xshift=-1.5pt,yshift=0.7pt] {a} (0,0);
 \draw[postaction={decorate}] (0,0) -- node[above,yshift=1.5pt] {y} (1,0);
 \draw[postaction={decorate}] (1,-1) -- node[below,yshift=-1.5pt] {y} (0,-1);
 \draw[postaction={decorate}] (1,-1) -- node[right,xshift=1.5pt,yshift=0.7pt] {b} (1,0);

 \draw (0.2,-0.4) -- (0.4,-0.2);
 \draw (0.2,-0.6) -- (0.6,-0.2);
 \draw (0.2,-0.8) -- (0.8,-0.2);
 \draw (0.4,-0.8) -- (0.8,-0.4);
 \draw (0.6,-0.8) -- (0.8,-0.6);
\end{tikzpicture}\vspace{0.3cm}

\begin{tikzpicture}[scale=1.5,baseline=-6mm,line width=0.5pt,decoration={
    markings,
    mark=at position 0.5 with {\arrow[xshift=3.333pt]{triangle 60}}},
    ]
 \filldraw[black] (0,0) circle (1pt) (0,-1) circle (1pt) (1,0) circle (1pt) (1,-1) circle (1pt);

 \draw[postaction={decorate}] (0,-1) -- node[left,xshift=-1.5pt,yshift=0.7pt] {b} (0,0);
 \draw[postaction={decorate}] (0,0) -- node[above,yshift=1.5pt] {x} (1,0);
 \draw[postaction={decorate}] (0,-1) -- node[below,yshift=-1.5pt] {x} (1,-1);
 \draw[postaction={decorate}] (1,0) -- node[right,xshift=1.5pt,yshift=0.7pt] {a} (1,-1);

 \draw (0.2,-0.4) -- (0.4,-0.2);
 \draw (0.2,-0.6) -- (0.6,-0.2);
 \draw (0.2,-0.8) -- (0.8,-0.2);
 \draw (0.4,-0.8) -- (0.8,-0.4);
 \draw (0.6,-0.8) -- (0.8,-0.6);

\end{tikzpicture}\qquad
\begin{tikzpicture}[scale=1.5,baseline=-6mm,line width=0.5pt,decoration={
    markings,
    mark=at position 0.5 with {\arrow[xshift=3.333pt]{triangle 60}}},
    ]
 \filldraw[black] (0,0) circle (1pt) (0,-1) circle (1pt) (1,0) circle (1pt) (1,-1) circle (1pt);

 \draw[postaction={decorate}] (0,-1) -- node[left,xshift=-1.5pt,yshift=0.7pt] {b} (0,0);
 \draw[postaction={decorate}] (0,0) -- node[above,yshift=1.5pt] {y} (1,0);
 \draw[postaction={decorate}] (1,-1) -- node[below,yshift=-1.5pt] {x} (0,-1);
 \draw[postaction={decorate}] (1,-1) -- node[right,xshift=1.5pt,yshift=0.7pt] {a} (1,0);

 \draw (0.2,-0.4) -- (0.4,-0.2);
 \draw (0.2,-0.6) -- (0.6,-0.2);
 \draw (0.2,-0.8) -- (0.8,-0.2);
 \draw (0.4,-0.8) -- (0.8,-0.4);
 \draw (0.6,-0.8) -- (0.8,-0.6);
\end{tikzpicture}
}

\newcommand{\BMFourAutomDihedral}{%
{\begin{tikzpicture}[>=stealth,scale=0.5, shorten >=2pt,node distance=4cm,on grid,auto,thick,every initial
by arrow/.style={*->}]
   (0,0) \node[state] (a)   {$a$};
   (2,0) \node[state] (bb) [below=of a] {$b^{-1}$};
   (0,2) \node[state] (b) [right=of a] {$b$};
   (2,2) \node[state] (aa) [right=of bb] {$a^{-1}$};

   \tikzstyle{every node}=[font=\footnotesize]
    \path[->]
    (a) edge  node [below] {$x|y$, $y|y^{-1}$, $y^{-1}|x$} (b)
    (a) edge [bend right] node [left] {$x^{-1}|x^{-1}$} (bb)
    (b) edge  [bend right] node [above] {$y|x^{-1}$, $x^{-1}|y^{-1}$, $y^{-1}|y$} (a)
    (b) edge  node [left]{$x|x$} (aa)
    (bb) edge  node [right]{$x|x$} (a)
    (bb) edge [bend right] node [below] {$x^{-1}|y$, $y^{-1}|x^{-1}$, $y|y^{-1}$} (aa)
    (aa) edge [bend right] node [right] {$x^{-1}|x^{-1}$} (b)
    (aa) edge  node [above]{$x|y^{-1}$, $y|x$, $y^{-1}|y$} (bb);
\end{tikzpicture}}
}

\newcommand{\ComplexFourRF}{%
\begin{tikzpicture}[scale=1.5,baseline=-6mm,line width=0.5pt,decoration={
    markings,
    mark=at position 0.5 with {\arrow[xshift=3.333pt]{triangle 60}}},
    ]
 \filldraw[black] (0,0) circle (1pt) (0,-1) circle (1pt) (1,0) circle (1pt) (1,-1) circle (1pt);

 \draw[postaction={decorate}] (0,-1) -- node[left,xshift=-1.5pt,yshift=0.7pt] {a} (0,0);
 \draw[postaction={decorate}] (0,0) -- node[above,yshift=1.5pt] {x} (1,0);
 \draw[postaction={decorate}] (1,-1) -- node[below,yshift=-1.5pt] {x} (0,-1);
 \draw[postaction={decorate}] (1,-1) -- node[right,xshift=1.5pt,yshift=0.7pt] {b} (1,0);

 \draw (0.2,-0.4) -- (0.4,-0.2); \draw (0.2,-0.6) -- (0.6,-0.2); \draw (0.2,-0.8) -- (0.8,-0.2); \draw
(0.4,-0.8) -- (0.8,-0.4); \draw (0.6,-0.8) -- (0.8,-0.6);
\end{tikzpicture}\qquad
\begin{tikzpicture}[scale=1.5,baseline=-6mm,line width=0.5pt,decoration={
    markings,
    mark=at position 0.5 with {\arrow[xshift=3.333pt]{triangle 60}}},
    ]
 \filldraw[black] (0,0) circle (1pt) (0,-1) circle (1pt) (1,0) circle (1pt) (1,-1) circle (1pt);

 \draw[postaction={decorate}] (0,-1) -- node[left,xshift=-1.5pt,yshift=0.7pt] {a} (0,0);
 \draw[postaction={decorate}] (0,0) -- node[above,yshift=1.5pt] {y} (1,0);
 \draw[postaction={decorate}] (1,-1) -- node[below,yshift=-1.5pt] {y} (0,-1);
 \draw[postaction={decorate}] (1,0) -- node[right,xshift=1.5pt,yshift=0.7pt] {b} (1,-1);

 \draw (0.2,-0.4) -- (0.4,-0.2);
 \draw (0.2,-0.6) -- (0.6,-0.2);
 \draw (0.2,-0.8) -- (0.8,-0.2);
 \draw (0.4,-0.8) -- (0.8,-0.4);
 \draw (0.6,-0.8) -- (0.8,-0.6);
\end{tikzpicture}\vspace{0.3cm}

\begin{tikzpicture}[scale=1.5,baseline=-6mm,line width=0.5pt,decoration={
    markings,
    mark=at position 0.5 with {\arrow[xshift=3.333pt]{triangle 60}}},
    ]
 \filldraw[black] (0,0) circle (1pt) (0,-1) circle (1pt) (1,0) circle (1pt) (1,-1) circle (1pt);

 \draw[postaction={decorate}] (0,-1) -- node[left,xshift=-1.5pt,yshift=0.7pt] {a} (0,0);
 \draw[postaction={decorate}] (1,0) -- node[above,yshift=1.5pt] {y} (0,0);
 \draw[postaction={decorate}] (0,-1) -- node[below,yshift=-1.5pt] {x} (1,-1);
 \draw[postaction={decorate}] (1,0) -- node[right,xshift=1.5pt,yshift=0.7pt] {a} (1,-1);

 \draw (0.2,-0.4) -- (0.4,-0.2);
 \draw (0.2,-0.6) -- (0.6,-0.2);
 \draw (0.2,-0.8) -- (0.8,-0.2);
 \draw (0.4,-0.8) -- (0.8,-0.4);
 \draw (0.6,-0.8) -- (0.8,-0.6);

\end{tikzpicture}\qquad
\begin{tikzpicture}[scale=1.5,baseline=-6mm,line width=0.5pt,decoration={
    markings,
    mark=at position 0.5 with {\arrow[xshift=3.333pt]{triangle 60}}},
    ]
 \filldraw[black] (0,0) circle (1pt) (0,-1) circle (1pt) (1,0) circle (1pt) (1,-1) circle (1pt);

 \draw[postaction={decorate}] (0,-1) -- node[left,xshift=-1.5pt,yshift=0.7pt] {b} (0,0);
 \draw[postaction={decorate}] (0,0) -- node[above,yshift=1.5pt] {x} (1,0);
 \draw[postaction={decorate}] (0,-1) -- node[below,yshift=-1.5pt] {y} (1,-1);
 \draw[postaction={decorate}] (1,0) -- node[right,xshift=1.5pt,yshift=0.7pt] {b} (1,-1);

 \draw (0.2,-0.4) -- (0.4,-0.2);
 \draw (0.2,-0.6) -- (0.6,-0.2);
 \draw (0.2,-0.8) -- (0.8,-0.2);
 \draw (0.4,-0.8) -- (0.8,-0.4);
 \draw (0.6,-0.8) -- (0.8,-0.6);
\end{tikzpicture}
}

\newcommand{\BMFourAutomSymmetric}{%
{\begin{tikzpicture}[>=stealth,scale=0.5, shorten >=2pt,node distance=4cm,on grid,auto,thick,every initial
by arrow/.style={*->}]
   (0,0) \node[state] (x)   {$a$};
   (2,0) \node[state] (yy) [below=of x] {$b$};
   (0,2) \node[state] (y) [right=of x] {$a^{-1}$};
   (2,2) \node[state] (xx) [right=of yy] {$b^{-1}$};

   \tikzstyle{every node}=[font=\footnotesize]
    \path[->]
    (x) edge  node [above] {$y^{-1}|x$, $x^{-1}|y$} (y)
    (x) edge [bend right] node [left] {$x|x^{-1}$} (yy)
    (y) edge  [bend right] node [above] {$y|x^{-1}$, $x|y^{-1}$} (x)
    (y) edge  node [above] {$x^{-1}|x$}  (xx)
    (yy) edge  node [below] {$x^{-1}|x$} (x)
    (yy) edge [bend right] node [below] {$y^{-1}|x^{-1}$, $x|y$} (xx)
    (xx) edge [bend right] node [right] {$x|x^{-1}$} (y)
    (xx) edge  node [below] {$x^{-1}|y^{-1}$, $y|x$} (yy)
    (xx) edge [bend right] node [below, sloped] {$y^{-1}|y$} (x)
    (x) edge [bend right] node [above, sloped] {$y|y^{-1}$} (xx)
    (y) edge  [bend right] node [above, sloped] {$y^{-1}|y$} (yy)
    (yy) edge [bend right] node [below, sloped] {$y|y^{-1}$} (y);
\end{tikzpicture}}
}

\maketitle

\begin{abstract}
The first example of a non-residually finite group in the classes of finitely presented small-cancelation groups, automatic groups, and CAT(0) groups was constructed by Wise as the fundamental group of a complete square complex (CSC for short) with twelve squares.
At the same time, Janzen and Wise proved that CSCs with at most three squares, five or seven squares have residually finite fundamental group. The smallest open cases were CSCs with four squares and directed complete $\VH$ complexes with six squares. We prove that the CSC with four squares studied by Janzen and Wise has a non-residually finite fundamental group. In particular, this gives a non-residually finite CAT(0) group isometric to $F_2\times F_2$. For the class of complete directed $\VH$ complexes, we prove that there are exactly two complexes with six squares having a non-residually finite fundamental group. In particular, this positively answers to a question of
Wise on whether the main example from his PhD thesis is non-residually finite. As a by-product, we get finitely presented torsion-free simple groups which decompose into an amalgamated free product of free groups $F_7*_{F_{49}}F_7$.

Our approach relies on the connection between square complexes and automata discovered by Glasner and Mozes, where complete $\VH$ complexes with one vertex correspond to bireversible automata. We prove that the square complex associated to a bireversible automaton with two states or over the binary alphabet generating an infinite automaton group has a non-residually finite fundamental group. We describe automaton groups associated to CSCs with four squares and get two simple automaton representations of the free group $F_2$ and the first automaton representation of the free product $C_3*C_3$.

\vspace{0.2cm}\textit{2010 Mathematics Subject Classification}: 20F65, 20M35, 20E08

\end{abstract}


\section{Introduction}

One of the outstanding open problems in geometric group theory is whether word-hyperbolic groups are
residually finite, which is important in understanding of the topology of hyperbolic spaces. Before the introduction of word-hyperbolic groups, Schupp and later Pride, Gersten, and others asked about the residual properties of groups in a neighborhood of word-hyperbolic groups: for finitely presented small-cancelation groups, automatic groups, CAT(0) groups. The first example of a non-residually finite group belonging to each of these classes was constructed by Wise in his dissertation \cite{Wise:PhD}. Shortly after that Burger and Mozes \cite{BM:FP_simple,BM:LatticesProductTrees} constructed finitely presented torsion-free simple groups. All these examples
are the fundamental groups of complete square complexes (non-positively curved square complexes covered by the direct product of two trees).

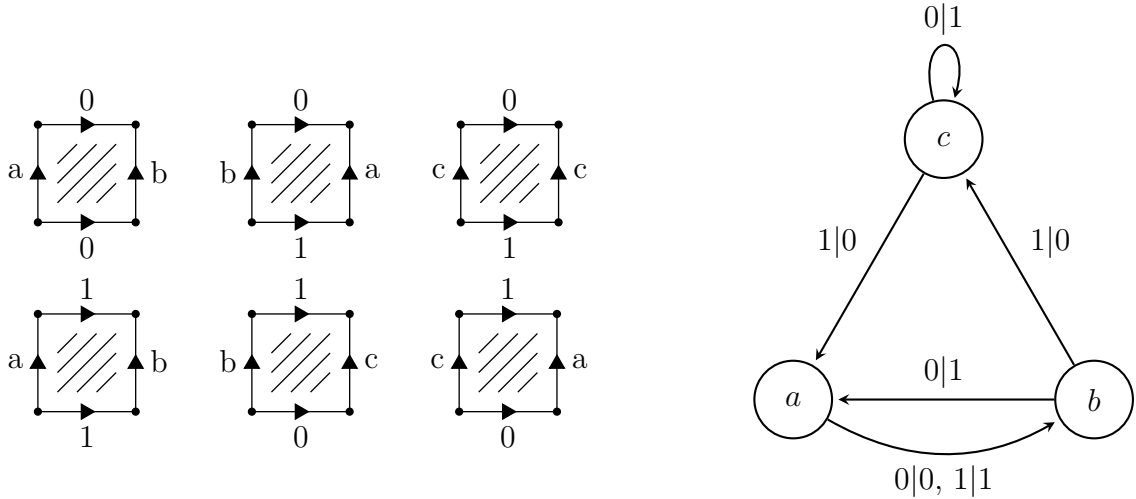
\begin{figure}
\begin{minipage}[b]{0.60\textwidth }
\begin{center}
\[\arraycolsep=1.4pt\def\arraystretch{2.2}
\begin{array}{c}
\WangTile{a}{0}{0}{b}\quad  \WangTile{b}{0}{1}{a}\quad  \WangTile{c}{0}{1}{c}\qquad\\
\WangTile{a}{1}{1}{b}\quad  \WangTile{b}{1}{0}{c}\quad  \WangTile{c}{1}{0}{a} \qquad
\end{array}
\]
\end{center}\vspace{0.2cm}
\end{minipage}\hfill
\begin{minipage}[b]{0.39\textwidth }
\begin{center}
{\begin{tikzpicture}[>=stealth,scale=0.5,shorten >=2pt, thick,every initial by arrow/.style={*->}]
  \node[state] (a) at (0, 0)   {$a$};
  \node[state] (b) at +(0: 8)  {$b$};
  \node[state] (c) at +(60: 8) {$c$};
    \path[->,scale=4]
    (b) edge  node [above]{$0|1$} (a)
    (b) edge  node [above right]{$1|0$} (c)
    (a) edge  [bend right] node [below]{$0|0$, $1|1$} (b)
    (c) edge  node [above left]{$1|0$} (a)
    (c) edge  [loop above] node {$0|1$} (c);
\end{tikzpicture}}
\end{center}
\end{minipage}
\caption{The Wise complex $\SqComplex_W$ and the Aleshin automaton}
\label{fig_WiseCompl_AleshynAut}
\end{figure}

The construction of Wise is based on a special square complex $\SqComplex_W$, which is called the main
example in \cite{Wise:PhD,Wise:CSC}. This complex is obtained by gluing the six unit squares shown in
Figure~\ref{fig_WiseCompl_AleshynAut}. Wise studied tiling properties of these squares and proved that
they admit a special non-periodic tiling of the plane with periodically labeled axes (an anti-torus). This observation was the key ingredient in several interesting examples (see \cite{HsuWise:SquareGroups,HsuWise:FixedSubgr,Wise:PeriodicFlat}). In particular,
Wise proved that the subgroup of $\pi_1(\SqComplex_W)$ generated by the loops $a,b,c$ is not separable, and therefore the amalgamated free product $\pi_1(\SqComplex_W)*_{\langle a,b,c\rangle}\pi_1(\SqComplex_W)$
is a non-residually finite CAT(0) group. The question was
raised \cite[Problem~10.19]{Wise:CSC} whether $\pi_1(\SqComplex_W)$ itself is non-residually finite. Below we give a
positive answer to this question as a corollary of a more general statement.

Our approach and initial interest in square complexes lie through the theory of
automaton groups, which is another fascinating topic of modern group theory. The theory of automaton groups deals with a special class of automata-transducers (see example in Figure~\ref{fig_WiseCompl_AleshynAut}). The relation between such automata and square complexes was discovered by Glasner and Mozes in \cite{GlMozes:AutomataSqComp}. Let $A$ be an automaton with the set of states $S$ and an input-output alphabet $X$. The square complex $\SqComplex_A$ associated to $A$ is obtained by gluing the squares given by the arrows in $A$ as shown in Figure~\ref{fig_arrow_square}. The complex $\SqComplex_A$ belongs to the class of directed $\VH$ square complexes with one vertex. The fundamental group of $\SqComplex_A$ has presentation
\[
\pi_1(\SqComplex_A)=\langle S,X \ | \ sx=yt \ \mbox{ for each arrow $s\xrightarrow{x|y}t$ in $A$} \rangle,
\]
which immediately suggest that there should be a connection between geometric properties of $\SqComplex_A$ and the combinatorial structure of $A$. Indeed, Glasner and Mozes noticed that $\SqComplex_A$ is a CSC exactly when $A$ belongs to the class of bireversible automata introduced in \cite{MacNekrSush} in relation to commensurators of free groups. The bireversible automaton corresponding to the Wise complex $\SqComplex_W$ is shown in Figure~\ref{fig_WiseCompl_AleshynAut}.

Another algebraic object related to an automaton is the associated automaton group. An automaton group $G_A$ is defined by the action of an automaton $A$ on input words. Roughly speaking, a group is an automaton group if one can put an automaton structure on the group consistent with its group structure. The study of automaton groups was initially motivated by several examples, mainly the Grigorchuk group, that enjoy many fascinating properties: torsion, intermediate growth, amenable but not elementary amenable, non-uniformly exponential growth, finite width, just-infiniteness, etc. Further investigations showed that automaton groups naturally arise in diverse areas of mathematics (see \cite{GNS,gri_sunik:branching,self_sim_groups}). In this paper we show that automaton groups are useful in the study of square complexes as well. 


\begin{figure}
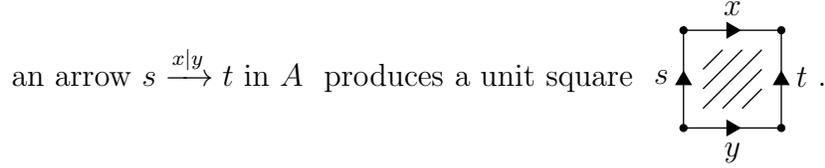

\[
\mbox{an arrow $s\xrightarrow{x|y}t$ in $A$ \ produces } \mbox{a unit square } \WangTile{$s$}{$x$}{$y$}{$t$}.
\]
\caption{A labeled unit square associated to an arrow in an automaton}\label{fig_arrow_square}
\end{figure}

{
\renewcommand{\thetheorem}{\ref{thm_infinite_separable}}
\begin{theorem}
Let $A=(X,S,\lambda)$ be a bireversible automaton.
 \begin{enumerate}
 \item If $G_A$ is finite, then $\pi_1(\SqComplex_A)$ is virtually a direct product of two free groups and
       therefore residually finite.
 \item If $G_A$ is infinite, then $\pi_1(\SqComplex_A)$ is not $\langle S\rangle$-separable and not $\langle X\rangle$-separable.
 \end{enumerate}
\end{theorem}
\addtocounter{theorem}{-1}
}

{
\renewcommand{\thetheorem}{\ref{thm_2state_NRF}}
\begin{theorem}
Let $A=(X,S,\lambda)$ be a bireversible automaton with two states or over the binary alphabet. If $G_A$ is infinite, then $\pi_1(\SqComplex_A)$ is non-residually finite.
\end{theorem}
\addtocounter{theorem}{-1}
}

The key ingredient in the proof of Theorem~\ref{thm_2state_NRF} is a nontrivial endomorphism of $\pi_1(\SqComplex_A)$ that is trivial on all states or all letters. The existence of such endomorphism follows from the fact that the automaton group $G_A$ contains a subautomaton nontrivially isomorphic to the original automaton $A$.
For example, for the Wise complex $\SqComplex_W$ the map
\[
\phi:\begin{array}{c}
       a\mapsto a, \\
       b\mapsto b, \\
       c\mapsto c,
     \end{array} \qquad
     \begin{array}{c}
       0\mapsto 01^{-1}01^{-1}0 \\
       1\mapsto 10^{-1}10^{-1}1
     \end{array}
\]
extends to an endomorphism of $\pi_1(\SqComplex_W)$ with $Fix(\phi)=\langle a,b,c\rangle$. Since $\pi_1(\SqComplex_W)$ is not $\langle a,b,c\rangle$-separable, it is non-residually finite; the element $(0^{-1}1)^4$ belongs to the intersection of finite index subgroups of $\pi_1(\SqComplex_W)$.


Interestingly, that the automaton associated to the Wise complex is well known in the theory of automaton groups. This automaton was constructed by Aleshin in \cite{A:free} in the first attempt to generate a free non-abelian group by (initial) automata. The proof was considered not complete, and the problem remained open for many years.
The first automaton realization of free non-abelian groups was made by Glasner and Mozes in \cite{GlMozes:AutomataSqComp} based on the connection with square complexes and Burger-Mozes groups. The smallest automaton generating a free group constructed by Glasner and Mozes has $6$ states over a $14$-letter alphabet.
Finally, Vorobets and Vorobets in \cite{Vorobets:Aleshyn} proved that the Aleshin automaton generates the free group of rank three. Surprisingly, among the hundreds of automata with three states over the binary alphabet only the Aleshin automaton generates a free non-abelian group (see \cite{Classif32}). Automaton realizations of free groups and free products of cyclic groups of order two are constructed in \cite{SavVor:FreeC2,SteinVor,Vorobets:Series}.

\begin{figure}
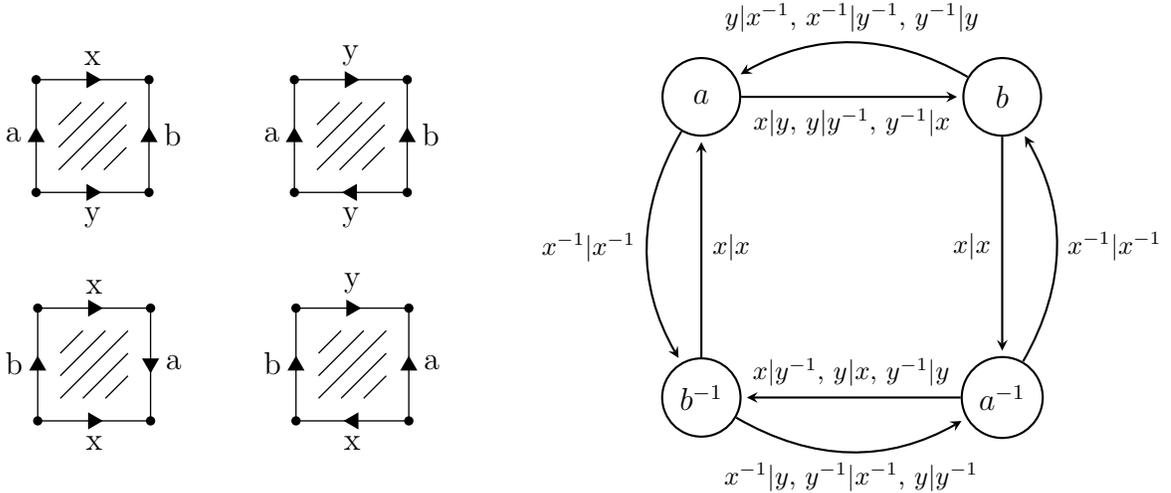

\begin{minipage}[b]{0.49\textwidth}
\begin{center}
\ComplexFourNRF
\end{center}\vspace{0.1cm}
\end{minipage}
\begin{minipage}[b]{0.49\textwidth}
\begin{center}
\BMFourAutomDihedral
\end{center}
\end{minipage}
\caption{A smallest CSC $\SqComplex_D$ with non-residually finite fundamental group and
the associated automaton with $G_A\cong C_3*C_3$ and $G_{\dual A}\cong F_2$}
\label{fig_complex_BM44_D4}
\end{figure}

In \cite{JW:SmIrrLattice} Janzen and Wise proved that CSCs with at most three squares, five or seven squares have residually finite fundamental group. Burger and Mozes in \cite{BM:LatticesProductTrees} proved that for every $n\geq 109$ and $m\geq 150$ there exists a complete $\VH$ complex with $n$ and $m$ elements in the vertical-horizontal decomposition having a non-residually finite (even virtually simple) fundamental group. The smallest open cases were CSCs with four squares and directed complete $\VH$ complexes with six squares. There are no bireversible automata with less than three states over the binary alphabet generating an infinite group, and only two bireversible automata with three states: the Aleshin automaton and the Bellaterra automaton (see \cite{Classif32}). By Theorems~\ref{thm_infinite_separable} and \ref{thm_2state_NRF} these two automata produce the smallest possible complete directed $\VH$ square complexes with non-residually finite fundamental groups.

The developed technique also works for non-directed $\VH$ complexes. Using computations with GAP, based on Rattaggi computations from \cite{Rattaggi:PhD}, we have checked that there are only two complete $\VH$ complexes $\SqComplex_D$ and $\SqComplex_S$ with four squares and one vertex that could have a non-residually finite group (for all other complexes the associated automata generate finite groups). These two complexes are shown in Figures~\ref{fig_complex_BM44_D4} and \ref{fig_complex_BM44_S4} together with the associated automata.
Both of these complexes were studied by Rattaggi, who conjectured that they have residually finite fundamental groups (see \cite[Section~4.10]{Rattaggi:ExaSQ}). Janzen and Wise in \cite{JW:SmIrrLattice} proved that $\SqComplex_D$ admits an anti-torus and, therefore, $\pi_1(\SqComplex_D)$ is a smallest irreducible lattice in the direct product of two trees. The question was raised whether $\pi_1(\SqComplex_D)$ is residually finite and in the theorem below we answer this question negatively. Interestingly, the corresponding automata provide new simple automaton representations of the free group $F_2$ and the first automaton representation of $C_3*C_3$.

{
\renewcommand{\thetheorem}{\ref{thm_SqComplex_D}}
\begin{theorem}
Let $\SqComplex_D$ be the square complex given by the four squares in Figure~\ref{fig_complex_BM44_D4} and
$A$ be the associated bireversible automaton. Then:
\begin{enumerate}
  \item $\pi_1(\SqComplex_D)$ is non-residually finite;
  \item $G_A\cong C_3*C_3$ \ and \ $G_{\dual A}\cong F_2$.
\end{enumerate}
\end{theorem}
\addtocounter{theorem}{-1}
}

Since the universal cover of $\SqComplex_D$ is the direct product of two regular trees of degree four, $\pi_1(\SqComplex_D)$ is a non-residually finite group isometric to $F_2\times F_2$. Therefore, the full group $C^{*}$-algebra of $\pi_1(\SqComplex_D)$ is not residually finite dimensional; this may be interesting in view of an open question whether the $C^{*}$-algebra of $F_2\times F_2$ is residually finite dimensional, which is equivalent to the Connes embedding conjecture.

In \cite{JW:SmIrrLattice} Janzen and Wise suggested that $\SqComplex_D$ is not a unique example of a complex with four squares that produce an irreducible lattice. We confirm this by proving that $\pi_1(\SqComplex_S)$ is irreducible as well, however, we do not know whether $\pi_1(\SqComplex_S)$ is residually finite. The freeness of automaton groups $G_A$ and $G_{\dual A}$ strongly suggest that all nontrivial normal subgroups of $\pi_1(\SqComplex_S)$ have finite index, what was conjectured in \cite[Conjecture~23]{Rattaggi:ExaSQ}. Moreover, the associated automaton $A$ is a smallest self-dual automaton that generates a free group.

{
\renewcommand{\thetheorem}{\ref{thm_SqComplex_S}}
\begin{theorem}
Let $\SqComplex_S$ be the square complex given by the four squares in Figure~\ref{fig_complex_BM44_S4} and
$A$ be the associated bireversible automaton. Then:
\begin{enumerate}
  \item $G_A\cong F_2$ \  and \ $G_{\dual A}\cong F_2$;
  \item $\langle a,x\rangle$, $\langle a,y\rangle$, $\langle b,x\rangle$ and $\langle b,y\rangle$ form anti-tori in $\pi_1(\SqComplex_S)$.
\end{enumerate}
\end{theorem}
\addtocounter{theorem}{-1}
}

\begin{figure}
\begin{minipage}[b]{0.49\textwidth}
\begin{center}
\ComplexFourRF
\end{center}\vspace{0.1cm}
\end{minipage}
\begin{minipage}[b]{0.49\textwidth}
\begin{center}
\BMFourAutomSymmetric
\end{center}
\end{minipage}
\caption{The square complex $\SqComplex_S$ and
the associated automaton with $G_A\cong G_{\dual A}\cong F_2$}
\label{fig_complex_BM44_S4}
\end{figure}

The paper is organized as follows. In Section~\ref{sect_AutomatonGroups} we recall basic facts about automaton groups (see \cite{GNS,gri_sunik:branching,self_sim_groups} for more details). In Section~\ref{sect_Automata_SqComplexes} we describe the connection between automaton groups and square complexes. The residual properties of $\pi_1(\SqComplex_A)$ are studied in Section~\ref{sect_ResidualFiniteness}. In Section~\ref{sect_CSC_4squares} we prove Theorems~\ref{thm_SqComplex_D} and \ref{thm_SqComplex_S}. In the last section, following the approach of Rattaggi \cite{Rattaggi:PhD}, we construct finitely presented torsion-free simple groups which decompose into an amalgamated free product $F_7*_{F_{49}}F_7$.

\vspace{0.3cm}\textbf{Acknowledgment}. The first author would like to thank D'Aniele Dangeli, Rostislav Grigorchuk, Dmytro Savchuk, and Yaroslav Vorobets for fruitful discussions. Also the authors would like to thank the developers of the program package \textrm{AutomGrp} \cite{AutomGrp} which has been used to perform many of the
computations described in this paper.

\section{Automata and automaton groups}\label{sect_AutomatonGroups}

\subsection{Automata}\label{subsect_Automata}

Let $X$ be a nonempty set and $X^{*}$ the free monoid over $X$. The elements of $X^{*}$ are finite words $x_1x_2\ldots x_n$, where $x_i\in X$ and $n\in\mathbb{N}$, together with the empty word. The length of a word $v=x_1x_2\ldots x_n$ is $|v|=n$. The set $X^n$ consists of all words of length $n$. By $X^{-1}$ we denote the set of formal elements $x^{-1}$ for $x\in X$. Let $X^{\pm *}$ be the set of all (non-reduced) words over $X\cup X^{-1}$. In contrast to general elements of $X^{\pm *}$, the nonempty elements of $X^{*}$ are called positive words over $X$. The free group generated by $X$ is denoted by $F_X$.

We consider complete deterministic automata-transducers (Mealy automata) with the same input and output alphabets.
Hence, in this article:
\begin{definition}
An \textit{automaton} is a triple $A=(S,X,\lambda)$, where $S$ and $X$ are nonempty sets and $\lambda:S\times X\rightarrow X\times S$ is an arbitrary map. The set $X$ is \textit{the alphabet} of input and output letters, the set $S$ is \textit{the set of states of $A$}.
\end{definition}
Finite automata have finitely many states and finite alphabet.
We identify an automaton $A=(S,X,\lambda)$ with a directed labeled graph on the vertex set $S$ with the following edges:
\[
s\xrightarrow{x|y}t \ \ \mbox{ whenever } \ \ \lambda(s,x)=(y,t).
\]
Note that for every $s\in S$ and every $x\in X$ there exists a unique arrow passing from $s$ and labeled by $x|y$ for some $y\in X$. The existence of such arrows indicates the completeness of $A$, while the uniqueness --- that $A$ is deterministic. In terms of automata theory, an arrow $s\xrightarrow{x|y} t$ means that if the automaton
is initialized at state $s$ and reads the letter $x$, then it outputs the letter $y$ and changes its active
state to $t$.

Two automata $A=(S_A,X_A,\lambda_A)$ and $B=(S_B,X_B,\lambda_B)$ are \textit{isomorphic} if there exist bijections $\phi:S_A\rightarrow S_B$ and $\psi:X_A\rightarrow X_B$ such that $\lambda_A(s,x)=(y,t)$ if and only if $\lambda_B(\phi(s),\psi(x))=(\psi(y),\phi(t))$. If $A$ and $B$ share the same alphabet $X=X_A=X_B$, then $A$ and $B$ are call \textit{$X$-isomorphic} if there exists an isomorphism that is trivial on $X$.

There are two standard operations over automata: taking dual automaton and taking inverse automaton for invertible automata. The \textit{dual of an automaton} $A=(S,X,\lambda)$ is the automaton $\dual(A)=(X^{-1},S^{-1},\delta)$, where
\[
\delta(x^{-1},s^{-1})=(t^{-1},y^{-1}) \quad \mbox{ if $\lambda(s,x)=(y,t)$},
\]
or in the graphical representation:
\[
x^{-1}\xrightarrow{s^{-1}|t^{-1}}y^{-1} \mbox{ \ in $\dual(A)$ \quad if \ }  s\xrightarrow{x|y}t \mbox{ \ in $A$}.
\]
The reason to put inverse sign will be clear in the next section. Since $A$ is complete and deterministic, the dual $\dual(A)$ is complete and deterministic too, so it is always a well-defined automaton. Note that $\dual(\dual(A))=A$.
We say that $A$ is \textit{self-dual} if $\dual(A)$ is isomorphic to $A$.

The \textit{inverse of an automaton} $A=(S,X,\lambda)$ is the tuple $\inv(A)=(S^{-1},X,\delta)$, where
\[
\delta(s^{-1},y)=(x,t^{-1}) \quad \mbox{ if $\lambda(s,x)=(y,t)$},
\]
or in the graphical representation:
\[
s^{-1}\xrightarrow{y|x}t^{-1} \mbox{ \ in $\inv(A)$ \quad if \ }  s\xrightarrow{x|y}t \mbox{ \ in $A$}.
\]
The $\inv(A)$ may be not deterministic/complete. The automaton $A$ is called \textit{invertible} if $\inv(A)$ is well-defined, which is equivalent to the following property:
\[
\mbox{for every $s\in S$ and $y\in X$ there exists an arrow $s\xrightarrow{x|y}t$ in $A$}
\]
for some $x\in X$ and $t\in S$. Note that $\inv(\inv(A))=A$.

Any automaton $A=(S,X,\lambda)$ can be naturally extended to an automaton $A^*=(S,X^*,\lambda^*)$ by consequently applying the rule:
\begin{align*}
s\xrightarrow{x_1|y_1}p \ \mbox{ and } \ p\xrightarrow{x_2|y_2}t \quad \ \mbox{ produce } \ \quad s\xrightarrow{x_1x_2|y_1y_2}t,
\end{align*}
which in automata theory corresponds to the consecutive processing of input strings of letters. Also, we can extend $A$ to an automaton $\leftidx{^*}{A}{}=(S^{*},X,\lambda^{*})$ by consequently applying the rule:
\begin{align*}
s_2\xrightarrow{y|z}t_2 \ \mbox{ and } \ s_1\xrightarrow{x|y}t_1 \quad \ \mbox{ produce } \quad \ s_2s_1\xrightarrow{x|z}t_2t_1,
\end{align*}
which is automata theory corresponds to the (left) composition of automata. By applying consequently these rules we get a well-defined automaton $\leftidx{^*}{A}{^*}=(S^{*},X^{*},\lambda^{*})$.

If $A$ is invertible, then we can consider the automaton $\leftidx{^{\pm*}}{A}{^*}=\leftidx{^*}{(A\cup\inv(A))}{^*}$ with the set of states $S^{\pm *}$ over the alphabet $X^{*}$. Note that the set of reduced words over $S$ spans a subautomaton, which produces an automaton structure on the free group $F_S$.

\subsection{Automaton groups}

Every automaton $A=(S,X,\lambda)$ produces two semigroup actions --- a left action $S^{*}\curvearrowright X^{*}$ and a right action $S^{*}\curvearrowleft X^{*}$ defined by the rule: for $g,h\in S^{*}$ and $u,v\in X^{*}$,
\[
g(u)=v \ \mbox{ and } \ (g)u=h \ \qquad \mbox{if $g\xrightarrow{u|v}h$ in $\leftidx{^*}{A}{^*}$},
\]
which is well-defined, because $\leftidx{^*}{A}{^*}$ is complete and deterministic. In other words, for each state $s\in S$, we put $s(x_1x_2\ldots x_n)=y_1y_2\ldots y_n$ for a path in $A$ of the form
\[
s\xrightarrow{x_1|y_1}s_2\xrightarrow{x_2|y_2}s_3\xrightarrow{x_3|y_3}\ldots\xrightarrow{x_n|y_n}s_{n+1},
\]
and the transformation of $X^{*}$ defined by a word over $S$ is exactly the left composition of the respective  transformations defined by the states. Similarly, the transformation of $S^{*}$ defined by a word over $X$ is the right composition of the transformations defined by the letters.

The transformations of $X^{*}$ defined by the states of $A$ are invertible if and only if $A$ is invertible. The transformation defined by a state $s^{-1}$ of $\inv(A)$ is inverse to the transformation defined by the state $s$ of $A$. The automaton $\leftidx{^{\pm*}}{A}{^*}$ defines an action of $S^{\pm*}$ on $X^{*}$, and the transformation of $X^{*}$ defined by a word $w\in S^{\pm*}$ is exactly the composition of transformations defined the states of $A$ and their inverses. Hence, every invertible automaton produces a natural action of the free group $F_S$ on $X^{*}$.

\begin{definition}
Let $A=(S,X,\lambda)$ be a finite invertible automaton. The quotient of $F_S$ by the kernel of its action on $X^{*}$ is called the \textit{automaton group} $G_A$.
\end{definition}

From another point of view, the automaton group $G_A$ is the group generated by the transformations of
$X^{*}$ defined by the states of $A$ under composition.

The dual automaton $\dual(A)=(X^{-1},S^{-1},\delta)$ produces a left action $(X^{-1})^{*}\curvearrowright
(S^{-1})^{*}$ and a right action $(X^{-1})^{*}\curvearrowleft (S^{-1})^{*}$. By taking formal inverses,
these two actions are translated to the right action $S^{*}\curvearrowleft X^{*}$ and the left action
$S^{*}\curvearrowright X^{*}$, respectively, associated to $A$:
\[
g\xrightarrow{u|v} h \ \mbox{ in $\leftidx{^*}{A}{^*}$} \quad \mbox{ if and only if } \quad u^{-1}\xrightarrow{g^{-1}|h^{-1}}v^{-1} \ \mbox{ in $\leftidx{^*}{\dual(A)}{^*}$}.
\]

The dual automaton $\dual(A)$ is invertible if and only if the right action $S^{*}\curvearrowleft X^{*}$ is
invertible. In this case, there is a natural automaton structure on $F_X$ over the alphabet $S$ and the
action of $F_X$ on the space $S^{*}$. The automaton group $G_{\dual(A)}$ is the quotient of $F_X$ by the
kernel of its action on $S^{*}$. The group $G_{\dual(A)}$ is generated by the transformation of $S^{*}$
defined by the letters in $X$ under composition.

Every automaton group $G$ admits a natural sequence of finite index subgroups
\[
G=\St_G(0)\geq\St_G(1)\geq\St_G(2)\geq\ldots,
\]
where $\St_G(n)=\{g\in G : g(v)=v \ \mbox{ for all $v\in X^n$}\}$ is the stabilizer of words of length $n$.
Since their intersection is trivial, every automaton group is residually finite. Another property of all
automaton groups is that they have solvable word problem.

\section{Automata and square complexes}\label{sect_Automata_SqComplexes}

In this section we describe the connection between automata and square complexes discovered by Glasner and Mozes in \cite{GlMozes:AutomataSqComp}. We give a somewhat different presentation with emphasis on combinatorial and algebraic properties.

\subsection{Square complexes}

A \textit{square complex} is a combinatorial $2$-complex whose $2$-cells are squares, i.e., they are attached by combinatorial paths of length four. We are interested in a special class of square complexes --- complete $\VH$ square complexes, introduced in \cite{Wise:PhD}.

A square complex is called a \textit{$\VH$ complex} if its $1$-cells can be partitioned into two classes $V$ and $H$ such that the attaching map of each $2$-cell alternates between the edges of $V$ and $H$.  If the attaching map of each $2$-cell preserves the orientation of the edges of $V$ and $H$, then the $\VH$ complex is called \textit{directed}. The Gromov link condition implies that a $\VH$ complex is non-positively curved if there are no double edges in the links of vertices.

A square complex is called a \textit{complete square complex} (CSC for short), if the link of each vertex is a complete bipartite graph. A natural example of a complete $\VH$ complex is a direct product of two graphs. Moreover, a square complex is complete if and only if its universal cover is a direct product of two trees (see \cite[Theorem~1.10]{Wise:PhD}), therefore, the fundamental group of a compact CSC acts freely and cocompactly on a CAT(0) space.

\subsection{Square complexes associated to automata}

Let $A=(S,X,\lambda)$ be a finite automaton. We associate to $A$ a set of Wang tiles, unit squares with labeled edges, as follows:
\[
W_A=\left\{ \WangTile{$s$}{$x$}{$y$}{$t$}  \mbox{ \ for each arrow $s\xrightarrow{x|y}t$ in $A$} \right\}.
\]
The set $W_A$ contains $\#S\cdot \#X$ squares, whose horizontal sides are labeled by letters from $X$,
while the vertical sides are labeled by states from $S$. In addition, all horizontal sides are oriented from
left to right, and all vertical edges are oriented from bottom to up.

The \textit{square complex $\SqComplex_A$ associated to $A$} has one vertex, a directed loop for every $s\in S$ and every $x\in X$, and a $2$-cell for every square in $W_A$. The complex $\SqComplex_A$ is a directed $\VH$ square complex. The fundamental group of $\SqComplex_A$ has finite presentation
\[
\pi_1(\SqComplex_A)=\langle S,X \ | \ sx=yt \ \mbox{ for each arrow $s\xrightarrow{x|y}t$ in $A$} \rangle.
\]
It is direct to see that an arrow $g\xrightarrow{v|u}h$ in $\leftidx{^*}{A}{^*}$ implies that $W_A$ admits a tiling of a finite rectangle such that its left side is labeled by $g$, the top side by $v$, the bottom side by $u$, and the right side by $h$, which produces the relation $gv=uh$ in $\pi_1(\SqComplex_A)$.

The next statement follows from the fact that we consider complete automata (see similar statements in \cite[Section~4.1]{Rattaggi:PhD} and \cite{DGKPR:BoundaryWang}).

\begin{proposition}
For any finite automaton $A$ the set of Wang tiles $W_A$ admits a periodic tiling of the plane and the group $\pi_1(\SqComplex_A)$ contains $\mathbb{Z}^2$ as a subgroup. Therefore, $\pi_1(\SqComplex_A)$ is never Gromov-hyperbolic.
\end{proposition}
\begin{proof}
Let us construct the directed graph $\Gamma$ on the vertex set $S\times X$, where we put the arrow $(s,x)\rightarrow (t,y)$ if $s\xrightarrow{x|y}t$ in $A$. Since $A$ is finite, the graph $\Gamma$ contains a directed cycle
\[
(s_1,x_1)\rightarrow (s_2,x_2)\rightarrow \ldots \rightarrow (s_n,x_n)\rightarrow (s_1,x_1).
\]
Therefore, we have the transitions in the automaton $\leftidx{^*}{A}{^*}$:
\[
s_n\ldots s_2s_1\xrightarrow{x_1|x_1} s_{1}s_n\ldots s_2\xrightarrow{x_2|x_2} s_2s_1\ldots s_3 \rightarrow\ldots\xrightarrow{x_n|x_n} s_n\ldots s_2s_1.
\]
Put $w=s_n\ldots s_2s_1$ and $u=x_1x_2\ldots x_n$. Then $\leftidx{^*}{A}{^*}$ contains a loop at $w$
labeled by $u|u$, which corresponds to the relation $wu=uw$ in $\pi_1(\SqComplex_A)$. It follows that $W_A$
admits a tiling of a rectangle with left/right labels $w$ and top/bottom labels $u$, which extends to a
periodic tiling of the plane.

Now we show that the subgroup $\langle w,u\rangle$ of $\pi_1(\SqComplex_A)$ is isomorphic to $\mathbb{Z}^2$. Since $\SqComplex_A$ is a directed $\VH$ complex, there exists a natural surjective homomorphism
\[
\phi: \pi_1(\SqComplex_A)\rightarrow\mathbb{Z}\times\mathbb{Z}
\]
which extends $s\mapsto (1,0)$ for all $s\in S$ and $x\mapsto (0,1)$ for all $x\in X$. Since we already know that $\langle w,u\rangle$ is abelian and $\phi(w)=(n,0)$, $\phi(u)=(0,n)$, the statement follows.
\end{proof}

It seems to be an interesting problem to develop a method which, given a finite automaton $A$, describes
all periodic tilings for the tileset $W_A$. These periodic tilings correspond to loops in the
automaton $\leftidx{^*}{A}{^*}$ labeled by $v|v$ for some $v\in X^{*}$. However, we do not see a nice ``finite'' description of all possible periodic tilings.

For general square complexes the following problem remains open.
\begin{problem}\label{probl_square_hyperb_Z2}
Is it true that if the fundamental group of a square complex is CAT(0) but not Gromov-hyperbolic, then it contains $\mathbb{Z}^2$ as a subgroup?
\end{problem}
This problem is a special case of a famous Gromov's question on whether each CAT(0) group which is not Gromov-hyperbolic contains $\mathbb{Z}^2$ as a subgroup. One of the approaches to get a negative answer to these problems would be to construct a set of Wang tiles which admits only non-periodic tilings of the plane with strong restrictions on side labels (see discussion in \cite[Section~4]{KarePapa:DetermAperiodic}, where the first $4$-way deterministic aperiodic tileset is constructed).

\subsection{Bireversible automata and complete square complexes}

Independently in automata theory, tiling theory and in the study of square complexes people came up to similar classes of objects with strong deterministic properties. In automata theory this property is called bireversibility. It was introduced in \cite{MacNekrSush} in relation to commensurators of free groups.

By applying the inverse and dual operations to any automaton $A$ we get eight (not necessary deterministic and complete) automata:
\begin{equation}\label{eqn_eight_automata}
A, \  \dual(A), \ \inv(A), \ \inv\dual(A), \ \dual\inv(A), \ \dual\inv\dual(A), \ \inv\dual\inv(A), \ \inv\dual\inv\dual(A)=\dual\inv\dual\inv(A).
\end{equation}
The arrows in each of these automata corresponding to an arrow in $A$ are shown in Figure~\ref{fig_arrows_dual_inverse}.
\begin{figure}
\begin{center}
\begin{tabular}{|rc|rc|}
  \hline
$A:$ & $s\xrightarrow{x|y}t$ & $\dual(A):$ & $x^{-1}\xrightarrow{s^{-1}|t^{-1}}y^{-1}$\\ \hline
$\inv(A):$ & $s^{-1}\xrightarrow{y|x}t^{-1}$ & $\dual\inv(A):$ & $y^{-1}\xrightarrow{s|t}x^{-1}$\\ \hline
$\dual\inv\dual(A):$ & $t\xrightarrow{x^{-1}|y^{-1}}s$ & $\inv\dual(A):$ & $x\xrightarrow{t^{-1}|s^{-1}}y$\\ \hline
$\dual\inv\dual\inv(A):$ & $t^{-1}\xrightarrow{y^{-1}|x^{-1}}s^{-1}$ & $\inv\dual\inv(A):$ & $y\xrightarrow{t|s}x$
\tabularnewline \hline
\end{tabular}\caption{Arrows in the eight automata obtained from $A$ by passing to the dual and inverse automata}\label{fig_arrows_dual_inverse}
\end{center}
\end{figure}

\vspace{-0.5cm}
\begin{definition}
A finite automaton $A$ is called \textit{bireversible} if all eight automata in (\ref{eqn_eight_automata}) are well-defined, i.e., complete and deterministic.
\end{definition}

Since we start with a complete and deterministic automaton $A$, all eight automata in (\ref{eqn_eight_automata}) are complete if and only if all of them are deterministic. Actually, since the complete and deterministic properties are preserved under passing to the dual automaton, an automaton $A$ is bireversible if and only if $\inv(A)$, $\inv\dual(A)$, and $\inv\dual\inv(A)$ are deterministic (equiv., complete), or in other words, $A$, $\dual A$ and $\dual\inv (A)$ are invertible.

Notice that all the arrows in Figure~\ref{fig_arrows_dual_inverse} produce the same relation in the fundamental groups of the corresponding square complexes:
\[
sx=yt,\qquad s^{-1}y=xt^{-1},\qquad x^{-1}t=sy^{-1},\qquad y^{-1}t^{-1}=s^{-1}x^{-1}.
\]
Therefore, the trivial map on $S$ and $X$ extends to an isomorphism between the square complexes associated to the eight automata.

The bireversibility of an automaton $A$ can be checked using a finite bipartite graph $\Gamma_A$ associate
to $A$. The vertex set of $\Gamma_A$ will be the disjoint union $(S\cup S^{-1})\cup (X\cup X^{-1})$. Each
arrow in $A$ contributes four edges in $\Gamma_A$:
\[
s\xrightarrow{x|y}t \mbox{ \, in $A$} \quad \Rightarrow \quad (s,x), \ (s^{-1},y), \ (t,x^{-1}), \ (t^{-1},y^{-1}) \mbox{ \, in $\Gamma_A$}.
\]
The four edges represent corresponding arrows in $A$, $\inv(A)$, $\inv\dual(A)$, and $\inv\dual\inv(A)$.
Since we want each of these automata to be complete, there should by edges $(s,x)$, $(s^{-1},y)$,
$(t,x^{-1})$, and $(t^{-1},y^{-1})$ for all $s,t\in S$ and $x,y\in X$. Therefore, $A$ is bireversible if
and only if $\Gamma_A$ is a complete bipartite graph. Note that $\Gamma_A$ is the link of the unique vertex
of $\SqComplex_A$.

Bireversibility has a nice interpretation in terms of the tileset $W_A$. Note that the deterministic
property of $A$ implies that the colors of two edges adjacent to the top left corner uniquely determines a
tile from $W_A$. The other three corners are responsible for deterministic properties of $\inv(A)$,
$\inv\dual(A)$, and $\inv\dual\inv(A)$. Therefore, the bireversibility of automata corresponds to the
$4$-way deterministic property of Wang tilesets (this property means that the colors of any two adjacent
edges uniquely determine a Wang tile).

\begin{proposition}
Let $A$ be a finite automaton. The following statements are equivalent:
\begin{enumerate}
  \item $A$ is bireversible;
  \item $\Gamma_A$ is a complete bipartite graph;
  \item $W_A$ is $4$-way deterministic;
  \item $\SqComplex_A$ is a complete square complex;
  \item $\SqComplex_A$ is non-positively curved;
  \item the universal cover of $\SqComplex_A$ is the direct product of two trees.
\end{enumerate}
\end{proposition}
\begin{proof}
The equivalence of items $1$, $2$, $3$ and $4$ is explained above. The equivalence of items $1$, $5$ and $6$ is
proved in \cite{GlMozes:AutomataSqComp}, and the equivalence of items $4$, $5$, $6$ follows from \cite{Wise:PhD}.
\end{proof}

Every bireversible automaton $A=(S,X,\lambda)$ can be extended to an automaton $A^{\pm}$ with the state set $S^{\pm 1}$ and the alphabet $X^{\pm 1}$, in which the arrows are given by the first column of
Figure~\ref{fig_arrows_dual_inverse}. Basically, $A^{\pm}$ is the union of $A$, $\inv(A)$,
$\dual\inv\dual(A)$, and $\dual\inv\dual\inv(A)$, while its dual $\dual(A^{\pm})$ is the union of
$\dual(A)$, $\dual\inv(A)$, $\inv\dual(A)$, and $\inv\dual\inv(A)$. Note that the automaton $A^{\pm}$ is
bireversible as well. Then we can naturally extend the state set of $A^{\pm}$ to words over $S^{\pm 1}$ and the alphabet to words over $X^{\pm 1}$ as in Section~\ref{subsect_Automata} and construct an automaton $\leftidx{^{\pm*}}{A}{^{\pm*}}=(S^{\pm*},X^{\pm*},\lambda^{*})$. The left group action $F_S\curvearrowright X^{*}$ associated to $A$ is extended to the left action $F_S\curvearrowright X^{\pm*}$ associated to $A^{\pm}$, while the right group action $S^{*}\curvearrowleft F_X$ is extended to the right action $S^{\pm*}\curvearrowleft F_X$. Note that the sets of reduced words $F_X\subset X^{\pm*}$ and $F_S\subset S^{\pm*}$ are invariant under these actions (they induce a subautomaton in $\leftidx{^{\pm*}}{A}{^{\pm*}}$).

\subsection{The fundamental group of $\SqComplex_A$ for bireversible automata}\label{subsect_fund_group_birever}

The next statement contains some basic properties of the fundamental groups of complexes $\SqComplex_A$ known for all complete $\VH$ square complexes with one vertex (see \cite{Wise:PhD}).

\begin{theorem}\label{prop_F_A_properties}
Let $A=(S,X,\lambda)$ be a bireversible automaton. 
\begin{enumerate}
  \item The group $\pi_1(\SqComplex_A)$ is torsion-free and CAT(0).

  \item The subgroups of $\pi_1(\SqComplex_A)$ generated by $S$ and $X$ are free of rank $\#S$ and $\#X$ respectively.

  \item(Normal forms) The group $\pi_1(\SqComplex_A)$ admits an exact factorization by its free subgroups $\langle S\rangle$ and $\langle X\rangle$. In particular, every element $\gamma\in \pi_1(\SqComplex_A)$ can be uniquely written in the form $\gamma=gv$ and in the form $\gamma=uh$ for $g,h\in\langle S\rangle$ and $v,u\in\langle X\rangle$.

\end{enumerate}
\end{theorem}

There is a nice direct connection between the transitions in a bireversible automaton $A$ and the two normal forms in $\pi_1(\SqComplex_A)$ given in Theorem~\ref{prop_F_A_properties}. The completeness of the bipartite graph $\Gamma_A$ means that the generators of $\pi_1(\SqComplex_A)$ satisfy
\[
S^{\pm 1}X^{\pm 1}=X^{\pm 1} S^{\pm 1},
\]
i.e., for any $s\in S^{\pm 1}$ and $x\in X^{\pm 1}$ there exists a unique pair $y\in X^{\pm 1}$ and $t\in S^{\pm 1}$ such that $sx=yt$ in $\pi_1(\SqComplex_A)$. This explains the normal form in $\pi_1(\SqComplex_A)$: given any word $\gamma$ over generators, we can move every $s\in S^{\pm 1}$ to the left and every $x\in X^{\pm 1}$ to the right in order to find the representation $\gamma=gv$ for $g\in\langle S\rangle$ and $v\in\langle X\rangle$. We can also move every $s\in S^{\pm 1}$ to the right and every $x\in X^{\pm 1}$ to the left, and obtain another normal form $\gamma=uh$ for $u\in\langle X\rangle$ and $h\in\langle S\rangle$. Every permutation of generators $sx=yt$ corresponds to a transition in $A^{\pm}$. Therefore, for reduced words $g,h\in F_S$ and $v,u\in F_X$ we have
\begin{equation}\label{eqn_relation_transition2}
gv=uh \mbox{ in $\pi_1(\SqComplex_A)$} \quad \mbox{ if and only if} \quad g\xrightarrow{v|u}h \ \mbox{ in $\leftidx{^{\pm*}}{A}{^{\pm*}}$}.
\end{equation}
In particular, we will frequently use the
following relation between the action of $F_S$ on $X^{\pm *}$ and the group $\pi_1(\SqComplex_A)$: for
$g\in F_S$ and $v\in F_X$,
\[
g(v)=v \ \ \mbox{ if and only if } \ \ v^{-1}gv\in F_S.
\]

\subsection{Automaton groups generated by bireversible automata}

Every bireversible automaton $A=(S,X,\lambda)$ gives rise to eight invertible automata
\[
A, \  \dual(A), \ \inv(A), \ \inv\dual(A), \ \dual\inv(A), \ \dual\inv\dual(A), \ \inv\dual\inv(A), \ \inv\dual\inv\dual(A)=\dual\inv\dual\inv(A),
\]
an automaton $A^{\pm}$, and two group actions: a left action $F_S\curvearrowright X^{\pm*}$ and a right action $S^{\pm*}\curvearrowleft F_X$.  The automaton group $G_{A^{\pm}}$ is the quotient of $F_S$ by the kernel of its action on $X^{\pm*}$. The subsets of positive words $X^{*}$ and negative words $(X^{-1})^{*}$ are invariant under the action of $F_S$. The corresponding restricted actions produce $G_A=G_{\inv(A)}$ and $G_{\dual\inv\dual(A)}=G_{\inv\dual\inv\dual(A)}$. Similarly, the dual automaton group $G_{\dual(A^{\pm})}$ is the quotient of $F_X$ by the kernel of its action on $S^{\pm*}$, while $G_{\dual(A)}=G_{\inv\dual(A)}$ and $G_{\dual\inv(A)}=G_{\inv\dual\inv(A)}$ are the quotients of the corresponding actions on $S^{*}$ and $(S^{-1})^{*}$.
The next statement shows that in this way we get just two groups $G_A$ and $G_{\dual(A)}$ and describes how to recover them from $\pi_1(\SqComplex_A)$.

\begin{theorem}\label{thm_G_A_and_pi_1}
Let $A=(S,X,\lambda)$ be a bireversible automaton. Then
\[
G_A=G_{\inv(A)}\cong G_{\dual\inv\dual(A)}=G_{\inv\dual\inv\dual(A)}\cong G_{A^{\pm}}\cong F_S/K, \\
\]
where $K$ is the maximal normal subgroup of $\pi_1(\SqComplex_A)$ that is contained in $F_S=\langle S\rangle$, and
\[
G_{\dual(A)}=G_{\inv\dual(A)}\cong G_{\dual\inv(A)}=G_{\inv\dual\inv(A)}\cong G_{\dual(A^{\pm})}\cong F_X/K_{\dual}.
\]
where $K_{\dual}$ is the maximal normal subgroup of $\pi_1(\SqComplex_A)$ that is contained in $F_X=\langle X\rangle$.
\end{theorem}
\begin{proof}
We show that the action of $G_{A^{\pm}}$ on $X^{*}$ is faithful. Let $g\in F_S$ act trivially on $X^{*}$, and let us show that $g$ acts trivially on $X^{\pm*}$. For any $v\in X^{*}$ there exists a unique $g_1\in F_S$ such that $gv=vg_1$ in $\pi_1(\SqComplex_A)$ and $g_1$ acts trivially on $X^{*}$. We can repeat this process and construct a sequence of elements $g_1,g_2,\ldots$ in $F_S$ such that
\[
gv=vg_1, \ \ g_1v=vg_2, \ \ g_2v=vg_3, \ \ \ldots \  \mbox{ in $\pi_1(\SqComplex_A)$},
\]
which corresponds to the directed path
\[
g\xrightarrow{v|v}g_1\xrightarrow{v|v}g_2\xrightarrow{v|v}g_3\xrightarrow{v|v}\ldots \quad \mbox{ in $\leftidx{^{\pm*}}{A}{^{\pm*}}$}.
\]
Since all $g_i$ have the same length, this sequence is eventually periodic. Moreover, the normal form in $\pi_1(\SqComplex_A)$ (or the deterministic properties of bireversible automata) implies that this sequence is periodic and there exists $n\in\mathbb{N}$ such that $g_n=g$. Hence, $gv^n=v^ng$ in $\pi_1(\SqComplex_A)$, which implies  $gv^{-n}=v^{-n}g$ and $g(v^{-1})=v^{-1}$. Therefore, $g$ acts trivially on $(X^{-1})^{*}$ as well and hence on $X^{\pm *}$. It follows that $G_A\cong G_{A^{\pm}}$. The other cases are analogous.

Let $K<\pi_1(\SqComplex_A)$ be the kernel of the action of $F_S\curvearrowright X^{\pm*}$ so that
$G_{A^{\pm}}=F_S/K$. Then $K$ is preserved under conjugation by elements of $F_S$. Since every element
of $K$ acts trivially on $F_X$, $K$ is preserved under conjugation by elements of $F_X$.
Therefore, $K$ is a normal subgroup of $\pi_1(\SqComplex_A)$. It is maximal among normal subgroups that are
contained in $F_S$, because all elements of such subgroups act trivially on words over $X^{\pm 1}$.
The dual case is analogous.
\end{proof}

The following statement is a combination of \cite[Proposition~2.2]{SavVor:FreeC2} and \cite[Proposition~1.2]{BM:LatticesProductTrees}. A group $\pi_1(\SqComplex_A)$ is called \textit{reducible} if it contains a finite index subgroup of the form $K\times H$ for $K<F_S$ and $H<F_X$.

\begin{corollary}\label{cor_Birev_Finite_Reducible}
Let $A=(S,X,\lambda)$ be a bireversible automaton. The following statements are equivalent:
\begin{enumerate}
  \item $G_A$ is finite;
  \item $G_{\dual A}$ is finite;
  \item $\pi_1(\SqComplex_A)$ is reducible.
\end{enumerate}
\end{corollary}
\begin{proof}
If $G_A$ is finite, then every orbit of the action $F_S\curvearrowright X^{*}$ contains at most $\#G_A$ elements. Therefore, every element $v\in X^{\pm*}$ belongs to a subautomaton of $\leftidx{^{\pm*}}{(\dual A)}{}$ with at most $\#G_A$ states. Since there are only finitely many different automata with a fixed number of states, the group $G_{\dual A}$ is finite. Hence, the items 1 and 2 are equivalent.

If $G_A$ and $G_{\dual A}$ are finite, then $K\times K_{\dual}$ has finite index in $\pi_1(\SqComplex_A)$. Conversely, if $\pi_1(\SqComplex_A)$ contains a finite index subgroup $K_1\times K_2$ with $K_1<F_S$ and $K_2<F_X$, then the kernels $K>K_1$ and $K_{\dual}>K_2$ have finite index in $F_S$ and $F_X$ respectively. Therefore, $G_A$ and $G_{\dual A}$ are finite.
\end{proof}

Theorem~\ref{thm_G_A_and_pi_1} suggests two ways to generate free groups by automata.

\begin{remark}\label{rem_just_infinite_free}
If for some bireversible automaton $A$ we had an infinite group $G_A$ and a just-infinite\footnote{An infinite group is called just-infinite if every nontrivial normal subgroup has finite index} group
$\pi_1(\SqComplex_A)$, then $G_A$ and $G_{\dual A}$ would be free groups freely generated by $S$ and $X$
respectively. However, $\pi_1(\SqComplex_A)$ cannot be just-infinite, because it projects onto
$\mathbb{Z}^2$. Nevertheless, this approach works for some non-directed $\VH$ complexes, where just-infinite
examples exist (see \cite{BM:LatticesProductTrees}) and produce free automaton groups.
\end{remark}

\begin{remark}\label{rem_nontriv_orbit}
It follows from the theorem that if $\{g_1,g_2,\ldots,g_n\}$ is an orbit of the action
$S^{\pm*}\curvearrowleft F_X$ and some $g_i$ represents a nontrivial element of $G_A$, then all
$g_1,\ldots,g_n$ represent nontrivial elements of $G_A$. In particular, if $G_A$ is infinite and $F_X$ acts
transitively on all reduced words of length $n$ for each $n\in\mathbb{N}$, then $G_A$ is a free group
freely generated by $S$. However, $F_X$ cannot have this transitivity property, because positive words over
$S$ are invariant under the action. Nevertheless, as in the previous remark, this approach works for some
non-directed $\VH$ complexes and was used by Glasner and Mozes in \cite{GlMozes:AutomataSqComp} to construct
the first examples of automata generating free groups.
\end{remark}

Besides Corollary~\ref{cor_Birev_Finite_Reducible}, it is not clear what is the relation between the groups $G_A$ and $G_{\dual(A)}$. For all examples that we know, the groups $G_A$ and $G_{\dual(A)}$ are either both finitely presented or both infinitely presented.

\begin{question}
Which bireversible automata generate finitely presented groups?
Is it true for a bireversible automaton $A$ that $G_A$ is virtually free if and only if $G_{\dual A}$ is virtually free? Is it true that groups generated by bireversible automata are linear?
\end{question}

For these questions it seems useful to consider the quotient of $\pi_1(\SqComplex_A)$ by the kernels of the
action of $\langle S\rangle$ on $\langle X\rangle$ and of the action of $\langle X\rangle$ on $\langle
S\rangle$. We get a group
\[
T_A=\pi_1(\SqComplex_A)/K\times K_{\dual}=G_A\cdot G_{\dual(A)},
\]
which admits an exact factorization by both automaton groups. The group $T_A$ brings information about
$G_A$ and $G_{\dual(A)}$, and their interconnection in $\pi_1(\SqComplex_A)$. It follows from the results
of Y.~Vorobets (private communication) that the group $T_A$ for the Aleshin and Bellaterra automata is finitely
presented. The first example of a bireversible automaton with non-finitely presented groups $G_A$ and
$G_{\dual(A)}$ is shown in Figure~\ref{fig_AutomatonZ3wrZ}; it is self-dual and generates the lamplighter
group $\mathbb{Z}_3\wr\mathbb{Z}$ (see \cite{BDR:LamplGr_Birev}).

\begin{figure}
\begin{center}
\begin{tikzpicture}[shorten >=1pt,node distance=1cm, on grid,auto,/tikz/initial text=,semithick,transform shape]
  \node[state] at (0,0) (a) {$a$};
  \node[state] at (2,3.46) (b) {$b$};
  \node[state] at (4,0) (c) {$c$};
    \path[->]
    (a) edge [in=220, out=250, loop] node  {$1|1$} (a)
    (b) edge [in=75, out=115, loop] node [above]  {$3|1$} (b)
    (c) edge [in=-60, out=-30, loop] node  {$2|1$} (c)
    (a) edge [bend left] node {$2|3$} (b)
    (b) edge [bend left] node {$1|3$} (c)
    (c) edge [bend left] node {$3|3$} (a)
    (a) edge  node {$3|2$} (c)
    (c) edge  node {$1|2$} (b)
    (b) edge  node {$2|2$} (a);
\end{tikzpicture}
\caption{A self-inverse self-dual automaton generating the lamplighter group $\mathbb{Z}_3\wr\mathbb{Z}$}\label{fig_AutomatonZ3wrZ}
\end{center}
\end{figure}

\begin{question}
For which bireversible automata $A$ is the group $T_A$ finitely presented?
\end{question}

\subsection{Anti-tori and tilings of the plane}\label{subsect_Anti_tori}

Let $A$ be a bireversible automaton and $W_{A^{\pm}}$ the set of Wang tiles associated to $A^{\pm}$. Since $W_{A^{\pm}}$ is complete and $4$-way deterministic, all possible tilings of the plane by $W_{A^{\pm}}$ can be described as follows: for every pair of sequences $(s_i)_{i\in\mathbb{Z}}$, $s_i\in S^{\pm 1}$ and $(x_i)_{i\in\mathbb{Z}}$, $x_i\in X^{\pm1 }$ there exists a unique tiling $t:\mathbb{Z}^2\rightarrow W_{A^{\pm}}$ of the plane such that the sequence $(s_i)_{i\in\mathbb{Z}}$ is read along the vertical axes, while the sequence $(x_i)_{i\in\mathbb{Z}}$ is read along the horizontal axes. For reduced sequences  this tiling corresponds to a plane in the universal cover of $\SqComplex_A$.

Let $a\in F_S$ and $b\in F_X$. Let us consider the tiling of the plane with the vertical axes labeled by $(a)_{i\in\mathbb{Z}}$ and the horizontal axes labeled by $(b)_{i\in\mathbb{Z}}$. If this tiling is periodic in both vertical and horizontal directions, then there exist nonzero $n,m\in\mathbb{Z}$ such that $a^nb^m=b^ma^n$ in $\pi_1(\SqComplex_A)$, which means that there is a torus in $\SqComplex_A$. If this tiling is not periodic, then we come to the following definition.

\begin{definition}
For $a\in F_S$ and $b\in F_X$, the subgroup $\langle a,b\rangle$ is called an
\textit{anti-torus} in $\pi_1(\SqComplex_A)$ if $a^nb^m\neq b^ma^n$ for all $n,m\in\mathbb{Z}\setminus\{0\}$.
\end{definition}

It seems to be an interesting and difficult problem, given a bireversible automaton $A$, describe all
anti-tori in $\pi_1(\SqComplex_A)$. This problem was studied in \cite{R:Anti-tori} for certain Burger-Mozes
groups, where anti-tori correspond to non-commuting pairs of Hamilton quaternions.

\begin{proposition}\label{prop_anti_infini_group}
Let $A=(S,X,\lambda)$ be a bireversible automaton. If $\pi_1(\SqComplex_A)$ admits an anti-torus, then $G_A$ is infinite.
\end{proposition}
\begin{proof}
Let $a\in F_S$ and $b\in F_X$ generate an anti-torus. We will prove that $a$ represents an element of infinite order in $G_A$. Let us assume that $a^n(v)=v$ for all $v\in X^{\pm*}$. Then we have a cycle in the automaton
\[
a^n\xrightarrow{b|b}a_2\xrightarrow{b|b}a_3\xrightarrow{b|b}\ldots\xrightarrow{b|b}a^n
\]
(see the proof of Theorem~\ref{thm_G_A_and_pi_1}). This means that $a^nb^m=b^ma^n$ in $\pi_1(\SqComplex_A)$ for some $m\in\mathbb{N}$, which contradicts the assumption of the statement.
\end{proof}

It is interesting whether the converse holds:

\begin{question}\label{probl_anti_tori}
Is it true that if $G_A$ is infinite then $\pi_1(\SqComplex_A)$ contains an anti-torus?
\end{question}

This question is related to the next one studied in the theory of automaton groups.

\begin{question}\label{probl_birev_Burnside}
Can a bireversible automaton $A$ generate an infinite torsion group $G_A$?
\end{question}

By Proposition~\ref{prop_anti_infini_group}, a positive answer to Question~\ref{probl_birev_Burnside}
implies a negative answer to Question~\ref{probl_anti_tori}. However, we expect a negative answer to
Question~\ref{probl_birev_Burnside} and a positive answer to Question~\ref{probl_anti_tori}. It was shown
in \cite{KPS:3birBirns} that a bireversible automaton with two and three states cannot generate an infinite
torsion group. Note, however, that there are famous examples of infinite torsion groups generated by
non-bireversible automata like the Grigorchuk group.


\subsection{Automata from square complexes}

Let $\SqComplex$ be a complete directed $\VH$ square complex with one vertex. Let $S$ and $X$ be the loops at the corresponding vertical-horizontal decomposition of $\SqComplex$. Since $\SqComplex$ is complete and directed, for any $s\in S$ and $x\in X$ there exists a unique pair $y\in X$ and $t\in S$ such that $sx=yt$ in $\pi_1(\SqComplex)$. This property defines an automaton $A$ over the alphabet $X$ with the set of states $S$. This automaton is bireversible and the square complex $\SqComplex_A$ is exactly $\SqComplex$.

The same construction works for non-directed complexes. Let $\SqComplex$ be a complete (not necessary directed) $\VH$ square complex with one vertex. Let $S$ and $X$ be the loops at the corresponding vertical-horizontal decomposition of $\SqComplex$. The completeness of $\SqComplex$ implies that for any $s\in S^{\pm 1}$ and $x\in X^{\pm 1}$ there exists a unique pair $y\in X^{\pm 1}$ and $t\in S^{\pm 1}$ such that $sx=yt$ in $\pi_1(\SqComplex)$.
As above, these relations produce a bireversible automaton $A_{\SqComplex}$ over the alphabet $X\cup X^{-1}$ and with the set of states $S\cup S^{-1}$. Two examples of non-directed square complexes and the associated automata are shown in Figures~\ref{fig_complex_BM44_D4} and \ref{fig_complex_BM44_S4}. Note that the square complex $\SqComplex_{A}$ associated to $A_{\SqComplex}$ is not $\SqComplex$, because when we construct a square complex from automaton we treat every state as an independent color and do not take into account inverse elements.
If we start with a directed complex $\SqComplex=\SqComplex_A$ associated to a bireversible automaton $A$, then $A_{\SqComplex}$ is exactly the automaton $A^{\pm}$ constructed in Section~\ref{subsect_fund_group_birever}.

\section{Residual finiteness of $\pi_1(\SqComplex_A)$}\label{sect_ResidualFiniteness}

In this section we study the residual properties of the group $\pi_1(\SqComplex_A)$ for bireversible
automata $A$. Our analysis follows the approach of Wise from \cite{Wise:PhD} and relies on the following
statement.

Let $H$ be a subgroup of a group $G$. Then $G$ is called \textit{$H$-separable} if $H$ is the intersection
of finite index subgroup of $G$. A group $G$ is residually finite if and only if it is separable with
respect to the trivial subgroup.

\begin{theorem}[{\cite[Theorem 7.2]{Wise:CSC}}]\label{thm_Wise_Fix_separable}
Let $\phi$ be an endomorphism of a finitely generated residually finite group $G$, and let $Fix(\phi)$ be
the subgroup of elements fixed by $\phi$. Then $G$ is $Fix(\phi)$-separable.
\end{theorem}
\begin{proof}
We include the proof for completeness. Let $a\in G\setminus Fix(\phi)$. Then $a^{-1}\phi(a)\neq e$ and
$a^{-1}\phi(a)\not\in N$ for some subgroup $N$ of finite index $n$. We can assume that $N$ is fully invariant
by passing to the intersection of all subgroups of index $\leq n$. Then $Fix(\phi)N$ has finite index and
does not contain $a$.
\end{proof}

\begin{corollary}[{\cite[Corollary 7.3]{Wise:CSC}}]\label{cor_Wise_double}
Let $D = G *_{H}G$ be the double of a group $G$ along its subgroup $H$. If $D$ is residually finite, then
$G$ is $H$-separable.
\end{corollary}

In \cite{Wise:PhD} Wise proved that $\pi_1(\SqComplex_W)$ is not $\langle a,b,c\rangle$-separable using the fact that it contains an anti-torus. We do not know which automata admit an anti-torus (see discussion in Section~\ref{subsect_Anti_tori}); instead,
we are using the following lemma, which relies on Zelmanov's solution of the restricted Burnside problem.
Note that the conclusion of the lemma immediately follows from the existence of an anti-torus.

Let $A=(X,S,\lambda)$ be a bireversible automaton. For $m\in\mathbb{N}$, let $P_m$ be the set of all pairs
$(x,y)$ of different letters $x,y\in X$ for which there exist $g\in S^{*}$ and $u\in X^{*}$ such that
$g^m(ux)=uy$ (this means there is a relation $g^mux=uyh$ in the group $\pi_1(\SqComplex_A)$ for some $h\in S^{*}$). Note that $P_m$ is empty only when $g^m$ is trivial in $G_A$ for every $g\in S^{*}$.

\begin{lemma}\label{lemma_P_m_nonempty}
Let $A=(X,S,\lambda)$ be a bireversible automaton. If $G_A$ is infinite, then $\cap_{m \in \mathbb{N}} P_m$ is nonempty.
\end{lemma}
\begin{proof}
Let us assume that this intersection is empty. Then, for each pair of different letters $(x, y)$, there
exists $m_{xy}\in \mathbb{N}$ such that $(x,y)\notin P_{m_{xy}}$. Let $m$ be the product of all these
numbers $m_{xy}$. Note that if $n$ divides $m$, then $P_m\subset P_n$, because we can rewrite the equality
$g^m(ux)=uy$ as $(g^{m/n})^n(ux)=uy$. Since $m_{xy}$ divides $m$ for all pairs $(x,y)$, the set $P_m$ is
empty. It follows that the element $g^m$ is trivial in $G_A$ for every $g\in S^{*}$. Hence the group $G_A$
has finite exponent. Since $G_A$ is finitely generated and residually finite, it should be finite by the
solution of the restricted Burnside problem. We get a contradiction.
\end{proof}

\begin{remark}
For automata over the binary alphabet $X=\{0,1\}$ it is straightforward to see that the infiniteness of $G_A$ implies that $\cap_{m \in \mathbb{N}} P_m=\{(0,1),(1,0)\}$, and we do not need to rely on the restricted Burnside problem.
\end{remark}

\begin{theorem}\label{thm_infinite_separable}
Let $A=(X,S,\lambda)$ be a bireversible automaton.
 \begin{enumerate}
 \item If $G_A$ is finite, then $\pi_1(\SqComplex_A)$ is virtually a direct product of two free groups and therefore residually finite.
 \item If $G_A$ is infinite, then $\pi_1(\SqComplex_A)$ is not $\langle S\rangle$-separable and not $\langle X\rangle$-separable.
 \end{enumerate}
\end{theorem}
\begin{proof}
The first item follows from Corollary~\ref{cor_Birev_Finite_Reducible}.

Let $G_A$ be infinite. If $\pi_1(\SqComplex_A)$ is $\langle S\rangle$-separable, then for each $g\in \pi_1(\SqComplex_A)\setminus \langle S\rangle$ there exists a subgroup $H<\pi_1(\SqComplex_A)$ of finite index such that $\langle S\rangle<H$ and $g\notin H$. Then there exists a normal subgroup $N\lhd \pi_1(\SqComplex_A)$ of finite index such that $g\notin \langle S\rangle N$.

Let $(x,y)\in \cap_{m \in \mathbb{N}} P_m$. We are going to prove that the element $x^{-1}y$ of the group
$\pi_1(\SqComplex_A)$ belongs to $\langle S\rangle N$ for every normal subgroup $N$ of finite index in $\pi_1(\SqComplex_A)$. Since $x^{-1}y\not\in \langle S\rangle$ by Theorem~\ref{prop_F_A_properties}, it will follow that $\pi_1(\SqComplex_A)$ is not $\langle S\rangle$-separable.

Let $N<\pi_1(\SqComplex_A)$ be a normal subgroup of index $n$. Then for every $g\in \langle S\rangle$ the element $g^n$
belongs to $N$. Since $(x,y)\in P_n$, there exist words
$g,h\in S^{*}\subset\langle S\rangle$ and $u\in X^{*}$ such that $g^nux=uyh$ in $\pi_1(\SqComplex_A)$. Therefore
\[
x^{-1}y=h^{-1}(hx^{-1}u^{-1}g^nuxh^{-1})\in \langle S\rangle N,
\]
which completes the proof.
\end{proof}


\begin{corollary}
Let $A=(X,S,\lambda)$ be a bireversible automaton. The double $D_A=\pi_1(\SqComplex_A)*_{S}\pi_1(\SqComplex_A)$ of $\pi_1(\SqComplex_A)$ along $\langle S\rangle$ is residually finite if and only if $G_A$ is finite.
\end{corollary}
\begin{proof}
If $G_A$ is infinite, then $D_A$ is not residually finite by Theorem~\ref{thm_infinite_separable} and
Corollary~\ref{cor_Wise_double}.

If $G_A$ is finite, then $G_{A\cup\dual\inv\dual(A)}\cong G_{A}$ is finite. The group $D_A$ is the
fundamental group of the complex $\SqComplex_{A\cup \dual\inv\dual(A)}$, which is residually finite by
Theorem~\ref{thm_infinite_separable}.
\end{proof}

\begin{remark}
It follows from the proof of Theorem~\ref{thm_Wise_Fix_separable} that for each $(x,y)\in \cap_{m \in
\mathbb{N}} P_m$ the nontrivial element $(x^{-1}y)^{-1}\phi(x^{-1}y)=y^{-1}x\overline{x}^{-1}\overline{y}$ of
$\pi_1(\SqComplex_A)*_{S}\pi_1(\SqComplex_A)$ belongs to the intersection of finite index subgroups.
\end{remark}

\begin{remark}
Theorem~\ref{thm_infinite_separable} holds for the fundamental group of non-directed complete $\VH$ complexes $\SqComplex$ with one vertex and the associated bireversible automata, because the element $x^{-1}y$ from the proof remains nontrivial in $\pi_1(\SqComplex)$.
\end{remark}

The above theorem does not tell us when the group $\pi_1(\SqComplex_A)$ is residually finite, and the next question seems to be difficult:

\begin{question}
For which bireversible automata $A$ is the group $\pi_1(\SqComplex_A)$ residually finite?
\end{question}

We answer this question in some cases.

\begin{theorem}\label{thm_2state_NRF}
Let $A=(X,S,\lambda)$ be a bireversible automaton over an alphabet with two letters or with two states. If $G_A$ is
infinite, then $\pi_1(\SqComplex_A)$ is non-residually finite.
\end{theorem}
\begin{proof}
We can assume that $X=\{x,y\}$. There exists $n\in\mathbb{N}$ such that $(y^{-1}x)^n$ and $(x^{-1}y)^n$ act trivially on $S\cup S^{-1}$, i.e., $s(y^{-1}x)^{\pm n}s^{-1}\in\langle X\rangle$ for every $s\in S\cup S^{-1}$. Let us show that the map
\[
\phi(x)=x(y^{-1}x)^n, \ \phi(y)=y(x^{-1}y)^n \ \mbox{ and } \ \phi(s)=s \ \mbox{for all $s\in S$}
\]
extends to an endomorphism of $\pi_1(\SqComplex_A)$.

Since our alphabet is binary, a state $s\in S$ stabilizes $x$ if and only if it stabilizes $y$. Let $S^{+}$ be the set of all states that stabilize $x$ and $y$, and let $S^{-}$ be the set of all states that stabilize $x^{-1}$ and $y^{-1}$. Every defining relation for $\pi_1(\SqComplex_A)$ of the form $sx=xt$ or $sy=yt$ implies that $s\in S^{+}$ and $t\in S^{-}$, while a relation of the form $sx=yt$ or $sy=xt$ is possible only when $s\not\in S^{+}$ and $t\not\in S^{-}$. It follows that every $t\in S^{-}$ satisfies the relations $ty^{-1}x=y^{-1}xt_1$ and $tx^{-1}y=x^{-1}yt_2$ for some $t_1,t_2\in S^{-}$, while every $t\not\in S^{-}$ satisfies the relations $ty^{-1}x=x^{-1}yt_1$ and $tx^{-1}y=y^{-1}xt_2$ for some $t_1,t_2\not\in S^{-}$. Since $(y^{-1}x)^{\pm n}$ act trivially on $S$, we have relations
\begin{align*}
&t(y^{-1}x)^{\pm n}=(y^{-1}x)^{\pm n}t \ \mbox{ for every $t\in S^{-}$},\\
&t(y^{-1}x)^{\pm n}=(x^{-1}y)^{\pm n}t \ \mbox{ for every $t\not\in S^{-}$}.
\end{align*}
Now for every defining relation of the form $sx=xt$ and $sy=yt'$ (here $s\in S^{+}$ and $t,t'\in S^{-}$), there are relations
\[
sx(y^{-1}x)^n=xt(y^{-1}x)^n=x(y^{-1}x)^n t \ \mbox{ and } \ sy(x^{-1}y)^n=yt'(x^{-1}y)^n=y(x^{-1}y)^n t'.
\]
For defining relations of the form $sx=yt$ or $sy=xt'$ (here $s\not\in S^{+}$ and $t,t'\not\in S^{-}$),
there are relations
\[
sx(y^{-1}x)^n=yt(y^{-1}x)^n=y(x^{-1}y)^n t \ \mbox{ and } \ sy(x^{-1}y)^n=xt'(x^{-1}y)^n=x(y^{-1}x)^n t'.
\]
Thus, $\phi$ preserves all the defining relations of $\pi_1(\SqComplex_A)$ and induces an endomorphism of $\pi_1(\SqComplex_A)$.

Since $\phi(z)$ begins and ends on $z$ for each $z\in\{x,y,x^{-1},y^{-1}\}$, the image $\phi(w)$ of a reduced word $w\in\langle X\rangle$ is also reduced. The length of $\phi(w)$ is greater than the length of $w$; hence, no element in $\langle X\rangle$ can be fixed by $\phi$. Therefore, $Fix(\phi)=\langle S\rangle$ and $\pi_1(\SqComplex_A)$ is non-residually finite by Theorems~\ref{thm_Wise_Fix_separable} and \ref{thm_infinite_separable}.
\end{proof}

\begin{remark}
The element $(x^{-1}y)^{-1}\phi(x^{-1}y)=(x^{-1}y)^{-1}(x(y^{-1}x)^n)^{-1}(y(x^{-1}y)^n)=(x^{-1}y)^{2n}$ belongs to the intersection of finite index subgroups of $\pi_1(\SqComplex_A)$. The normal closure of $(x^{-1}y)^n$ and $(y^{-1}x)^n$ has infinite index in $\pi_1(\SqComplex_A)$.
\end{remark}

\begin{corollary}
Let $A=(X,S,\lambda)$ be a bireversible automaton. Suppose that each connected component of $A$ consists of two states and the automaton group $G_A$ is infinite. Then $\pi_1(\SqComplex_A)$ is non-residually finite.
\end{corollary}
\begin{proof}
Since $G_A$ is infinite, by Theorem \ref{thm_infinite_separable} we have that for some states $a,b \in S$
(possibly in different components) the element $a^{-1}b$ is not $\langle X\rangle$-separable. By applying
arguments from the proof of Theorem~\ref{thm_2state_NRF}, we can construct a nontrivial endomorphism $\phi_i$ for
each connected component $A_i=(X,S_i,\lambda)$ of $A$ considering it as a separate automaton. By the
construction, all $\phi_i$ map elements of $X$ to themselves. Hence, we could extend them to the unique
endomorphism $\phi$ of the whole group $\pi_1(\SqComplex_A)$. Similarly, the image of a reduced word
$w\in\langle S\rangle$ is also reduced and is longer than the initial word. So $a^{-1}b$ is not fixed, and
therefore $\pi_1(\SqComplex_A)$ is non-residually finite.
\end{proof}

\begin{example}\label{ex_Bellaterra}
The smallest automata that satisfy the conditions of Theorem~\ref{thm_2state_NRF} are the Aleshin
automaton $A$ and its friendly version --- the Bellaterra automaton $B$ shown in Figure~\ref{fig_BellaterraAut}. The corresponding endomorphism of $\pi_1(\SqComplex_A)$ for the Aleshin automaton is shown in Introduction. The Bellaterra automaton possesses an additional endomorphism. The states $\{a^{-1},b^{-1},c^{-1}\}$ of $B^{*}$ span a subautomaton
isomorphic to $B$. Hence, the map $\phi:\pi_1(\SqComplex_B)\rightarrow\pi_1(\SqComplex_B)$,
which replaces each letter from $S\cup S^{-1}$ by its inverse, is an automorphism of $\pi_1(\SqComplex_B)$
with $Fix(\phi)=\langle X\rangle$. The intersection of finite index subgroup of $\pi_1(\SqComplex_B)$ contains, for example, the elements $a^2b^{-2}$, $b^{2}c^{-2}$, $c^{2}a^{-2}$.
\end{example}

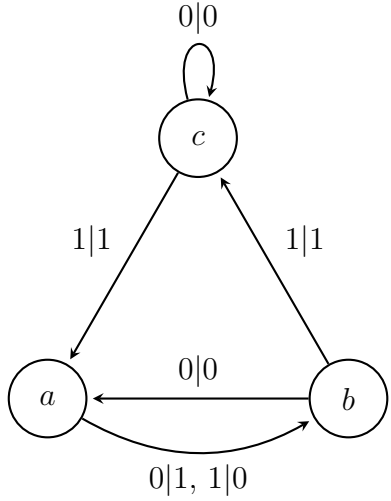
\begin{figure}
\begin{center}
{\begin{tikzpicture}[>=stealth,scale=0.5,shorten >=2pt, thick,every initial by arrow/.style={*->}]
  \node[state] (a) at (0, 0)   {$a$};
  \node[state] (b) at +(0: 8)  {$b$};
  \node[state] (c) at +(60: 8) {$c$};
    \path[->,scale=4]
    (b) edge  node [above]{$0|0$} (a)
    (b) edge  node [above right]{$1|1$} (c)
    (a) edge  [bend right] node [below]{$0|1$, $1|0$} (b)
    (c) edge  node [above left]{$1|1$} (a)
    (c) edge  [loop above] node {$0|0$} (c);
\end{tikzpicture}}
\end{center}
\caption{The Bellaterra automaton $B$ generating $G_B\cong C_2*C_2*C_2$}
\label{fig_BellaterraAut}
\end{figure}

It is an open question whether the finiteness problem for automaton groups is algorithmically solvable.
Probably not, because it is unsolvable for automaton semigroups (see \cite{Gill:SemiGrFinitAut}). However, the finiteness problem for automaton groups generated by bireversible automata over the binary alphabet (or with two states) is algorithmically solvable (see \cite{Kli:Finit2Birev}). This fact together with Theorems~\ref{thm_infinite_separable} and \ref{thm_2state_NRF} imply that there is an algorithm which for a given complete $\VH$ complex $\SqComplex$ with one vertex and two edges in the vertical or horizontal part verifies whether its fundamental group $\pi_1(\SqComplex)$ is residually
finite.

\begin{corollary}\label{cor_Subautom_NRF}
Let $A$ be a connected bireversible automaton generating an infinite automaton group. If there is a nontrivial $X$-isomorphism between $A$ and a subautomaton $B$ of the automaton $A^{*}$, then $\pi_1(\SqComplex_A)$ is
non-residually finite.
\end{corollary}
\begin{proof}
The isomorphism between $A$ and $B$ extends to an endomorphism $\phi:\pi_1(\SqComplex_A)\rightarrow \pi_1(\SqComplex_B)$ with fixed subgroup $Fix(\phi)$ containing $\langle X\rangle$. By repeating the arguments from the proof of
Theorem~\ref{thm_infinite_separable}, one can show that there is a pair of states $a,b$ in $A$ such that
$a^{-1}b\in Fix(\phi)N$ for every normal subgroup $N\lhd F_A$ of finite index. We just need to prove that
$a^{-1}b\not\in Fix(\phi)$. This will imply that $\pi_1(\SqComplex_A)$ is not $Fix(\phi)$-separable and, therefore, it is not residually finite by Theorem~\ref{thm_Wise_Fix_separable}.

Let $\phi(a)=v$ and $\phi(b)=u$. Since $B$ is connected, all words representing the states of $B$ have the
same length, in particular, $|v|=|u|$. Therefore, the equality $\phi(a^{-1}b)=v^{-1}u=a^{-1}b$ is possible
only when $v$ and $u$ have a common beginning $w$ so that $v=wa$ and $u=wb$. Since $\phi$ is nontrivial,
the word $w$ is nonempty. The first letter $s$ of $w$ is a state of $A$ or $A^{-1}$, and we can
assume $s\in A$. For any state $t$ of $A$ there is a path from $s$ to $t$, which induces two paths from
$wa$ and $wb$ to two different words which start with the letter $t$. It follows that for each state of
$A$ there are at least two states of $B$. This contradiction completes the proof.
\end{proof}

An automaton is minimal if different states define different transformations of words.

\begin{corollary}
Let $A$ be a non-minimal connected bireversible automaton generating an infinite automaton group. Then $\pi_1(\SqComplex_A)$ is non-residually finite.
\end{corollary}

The first automaton for which we cannot apply Corollary~\ref{cor_Subautom_NRF} is the self-dual automaton $A$ with $3$ states over a $3$-letter alphabet shown in Figure~\ref{fig_AutomatonZ3wrZ}. One can show that there is no subautomaton in $A^{*}$ that is $X$-isomorphic to $A$. We do not know whether its $\pi_1(\SqComplex_A)$ is residually finite.

\section{Complete square complexes with four squares}\label{sect_CSC_4squares}

Every automaton group $G_A$ possesses the following self-similarity property: for every $g\in G_A$ and $v\in X^{*}$ there exists a unique element $h\in G_A$ such that $g(vw)=g(v)h(w)$ for all $w\in X^{*}$. In the proof below we use notation $h=g|_v$.

\begin{theorem}\label{thm_SqComplex_D}
Let $\SqComplex_D$ be the square complex given by the four squares in Figure~\ref{fig_complex_BM44_D4} and
$A$ be the associated bireversible automaton. Then:
\begin{enumerate}
  \item $\pi_1(\SqComplex_D)$ is non-residually finite;
  \item $G_A\cong C_3*C_3$ \ and \ $G_{\dual A}\cong F_2$.
\end{enumerate}
\end{theorem}
\begin{proof}
We have $X=\{x,y,y^{-1},x^{-1}\}$ and $S=\{a,b,b^{-1},a^{-1}\}$.
Since $\pi_1(\SqComplex_D)$ contains an anti-torus (see \cite{JW:SmIrrLattice}), the automaton group $G_A$ is infinite by Proposition~\ref{prop_anti_infini_group}, and the group $\pi_1(\SqComplex_D)$ is not $\langle X\rangle$-separable by Theorem~\ref{thm_infinite_separable}. It is direct to check that the map defined by
\[
\phi(a)=a^4, \phi(b)=b^4 \ \mbox{ and } \ \phi(x)=x, \phi(y)=y
\]
extends to an endomorphism of $\pi_1(\SqComplex_D)$ with $Fix(\phi)=\langle X\rangle$. Therefore, $\pi_1(\SqComplex_D)$ is non-residually finite by Theorem~\ref{thm_Wise_Fix_separable}. Further we indicate nontrivial elements in the intersection of its finite index subgroups.

Let us describe the group $G_{\dual A}$. One can directly check the following crucial property: any nontrivial orbit of the action of an element $g\in \St_1(G_{A})$ on $X^2$
consists of three elements, namely, these orbits are
\[
\{xx,xy,xy^{-1}\}, \{yx,yy,yx^{-1}\}, \{y^{-1}x,y^{-1}y^{-1},y^{-1}x^{-1}\}, \{x^{-1}y,x^{-1}y^{-1},x^{-1}x^{-1}\}.
\]
It follows that, for any $g\in G_A$ and $v\in X^{*}$ with the property $g(v)=v$ and $g|_v\in \St_1(G_A)$, if $g(vz_1z_2)\neq vz_1z_2$ for some $z_1,z_2\in X$, then the orbit of $vz_1z_2$ under $g$ consists of three words $vz_1z_3$ for $z_3\in X$, $z_3\neq z_1^{-1}$. In particular, if $g\in \St_n(G_A)$ and $g(vz)\neq vz$ for some word $v\in X^n$, $z\in X$ (note that such $vz$ is necessary a (freely) reduced word), then the orbit of $vz$ under $g$ contains all three reduced words of length $n+1$ with prefix $v$.

We will show by induction on $n$ that the group $G_A$ acts transitively on the set of (freely) reduced words over $X$
of length $n$ for each $n\in\mathbb{N}$. Let us assume that the statement holds for the words of length
$n$. Since the group $G_A$ is infinite, there exists $g\in G_A$ such that $g\in \St_n(G_A)$ and $g(vz_1)\neq vz_1$ for some $vz_1\in X^{n+1}$. Note that $vz_1$ and $g(vz_1)$ are reduced words. Then the orbit of $vz_1$ under $g$ consists of exactly three reduced words of length $n+1$ with prefix $v$. Now let $wz_2\in X^{n+1}$ be an arbitrary
reduced word. By induction hypothesis there exists $h\in G_A$ such that $h(w)=v$. Then the word $h(wz_2)$
is reduced and has prefix $v$; therefore, it belongs to the orbit of $vz_1$ under $g$. This means that $vz_1$
and $wz_2$ belong to the same $G_A$-orbit. Our claim is proved.

We are ready to show that $G_{\dual A}$ is freely generated by $X$. Let $w$ be any reduced word over $X$. The transitivity of the action of $G_{A}$ on reduced words over $X$ implies that if $w$ represents the trivial element of $G_{\dual A}$, then all reduced word of length $|w|$ represent the trivial element (see Remark~\ref{rem_nontriv_orbit}). However, this would mean that $G_{\dual A}$ is finite, but it is not.

Similarly, we show that $G_{A}=\langle a\rangle * \langle b\rangle=C_3*C_3$. One checks directly that
$a^3=b^3=e$ in $G_A$. We have to prove that every word of the form $[a^{\pm 1}]b^{\pm 1}a^{\pm 1}\ldots
a^{\pm 1}[b^{\pm 1}]$ represents a nontrivial element of $G_A$. Note that the generators $x$ and $y$ map $a^{\pm 1}$ to $b^{\pm 1}$ and $b^{\pm 1}$ to $a^{\pm 1}$, and therefore $G_{\dual A}$ preserves the set of alternating words.  As above, it is enough to show that the
dual group $G_{\dual A}$ acts transitively on all such alternating words of fixed length $n$ for each
$n\in\mathbb{N}$. The proof goes in the same way as above and relies on the following property of
$G_{\dual A}$: any nontrivial orbit of the action of an element $g\in \St_1(G_{\dual A})$ on $S^2$
consists of two elements, namely, these orbits are
\begin{equation}\label{eqn_orbits_ThSqD}
\{ab,ab^{-1}\}, \ \{ba,ba^{-1}\}, \ \{b^{-1}a,b^{-1}a^{-1}\}, \ \{a^{-1}b,a^{-1}b^{-1}\}.
\end{equation}
Then, for any element $g\in G_{\dual A}$ and a word $v\in S^{*}$ with the property $g(v)=v$ and $g|_v\in
St_1(G_{\dual A})$, if $g(vz_1z_2)\neq vz_1z_2$ for some $z_1,z_2\in S$, then $g(vz_1z_2)=vz_1z_2^{-1}$
and $g(vz_1z_2^{-1})=vz_1z_2$. Let us assume by induction that $G_{\dual A}$ acts transitively on
alternating words of length $n$. Since $G_{\dual A}$ is infinite, there exist $g\in \St_n(G_{\dual A})$ and
$vz_{1}\in S^{n+1}$ such that $g(vz_{1})\neq vz_{1}$; therefore, $g(vz_{1})=vz_{1}^{-1}$. Note that
$vz_{1}$ is necessary an alternating word, since otherwise $vz_{1}$ could be represented by a shorter word
and would be fixed by $g$. Now let $wz_2\in S^{n+1}$ be an arbitrary alternating word of length $n+1$. By
induction hypothesis, there exists $h\in G_{\dual A}$ such that $h(w)=v$. Then $h(wz_2)$ is equal to either $vz_1$ or $vz_1^{-1}$. In any case, we get that $vz_1$ and $wz_2$ belong to the same $G_{\dual A}$-orbit. Our claim is proved.

Now we indicate an element in the intersection $N$ of finite index subgroup of $\pi_1(\SqComplex_D)$ and
compute the intersection $\cap_{m\in\mathbb{N}}P_m$ defined in Lemma~\ref{lemma_P_m_nonempty} for the dual automaton $\dual
A$. For every $m\in\mathbb{N}$ there exists an element $g\in G_{\dual A}$ such that $g^m\neq e$ and
$g^m\in \St_1(G_{\dual A})$. Then $g^m|_v\in \St_1(G_{\dual A})$ and $g^m|_v\not\in \St_2(G_{\dual A})$ for
some word $v\in S^{*}$. The only possible orbits of the action of $g^m|_v$ on $S^2$ are listed in
(\ref{eqn_orbits_ThSqD}), which imply that $P_m$ may contain only pairs $\{a,a^{-1}\}$ and $\{b,b^{-1}\}$, and one of
these pairs belongs to $P_m$. Therefore, either $a^{-2}\phi(a^2)=a^6$ or
$b^{-2}\phi(b^2)=b^6$ belongs to $N$. However, there is an automorphism
$\gamma:\pi_1(\SqComplex_D)\rightarrow \pi_1(\SqComplex_D)$ defined by
\[
\gamma(a)=b, \ \gamma(b)=a, \ \gamma(x)=x^{-1}, \ \gamma(y)=y^{-1}.
\]
It follows that if one of $a^6$ and $b^6$ is in $N$, then they both are. Another implication is that
$\cap_{m \in \mathbb{N}} P_m = \{\{a,a^{-1}\}, \{b,b^{-1}\}\}$.
\end{proof}

\begin{remark}
Since $a^3$ and $b^3$ are relations in $G_A$ and $G_A$ is infinite, the normal closure of $a^3$ and $b^3$ in $\pi_1(\SqComplex_D)$ has infinite index. Also, the intersection of finite index subgroups of $\pi_1(\SqComplex_D)$ has infinite index (see \cite[Theorem~34]{Rattaggi:ExaSQ}).
\end{remark}

\begin{theorem}\label{thm_SqComplex_S}
Let $\SqComplex_S$ be the square complex given by the four squares in Figure~\ref{fig_complex_BM44_S4} and
$A$ be the associated bireversible automaton. Then:
\begin{enumerate}
  \item $A$ is self-dual and  $G_{A}\cong F_2$;
  \item $\langle a,x\rangle$, $\langle a,y\rangle$, $\langle b,x\rangle$ and $\langle b,y\rangle$ form anti-tori in $\pi_1(\SqComplex_S)$.
\end{enumerate}
\end{theorem}
\begin{proof}
We have $X=\{x,y,y^{-1},x^{-1}\}$ and $S=\{a,b,b^{-1},a^{-1}\}$. Notice that the map
\[
a\mapsto y, \quad b\mapsto x, \quad x\mapsto b^{-1}, \quad y\mapsto a^{-1}
\]
is an isomorphism between $A$ and $\dual A$. Therefore, $G_A\cong G_{\dual A}$.

Let us show that $G_{\dual A}$ is freely generated by $X$. As in Theorem~\ref{thm_SqComplex_D}, it is sufficient to prove that $G_A$ acts transitively on reduced words over $X$ of length $n$ for each $n\in\mathbb{N}$. Since $G_A$ enjoys the crucial property from the proof of Theorem~\ref{thm_SqComplex_D} about the structure of $G_A$-orbits on $X^2$, the proof goes the same way once we show that $G_A$ is infinite.
We notice another cute property of $G_A$:
the elements of $S_3=\{a^{3},b^{3},b^{-3},a^{-3}\}$ preserve the set $X_3=\{x^{3},y^{3},y^{-3},x^{-3}\}$, and  the map
\[
a\mapsto a^{-3}, \quad b\mapsto b^{-3}, \quad x\mapsto x^3, \quad y\mapsto y^3
\]
is an isomorphism between $A$ and the subautomaton of $\leftidx{^*}{A}{^*}$ formed by the states $S_3$ over the alphabet $X_3$. Since $a^4(x)=x$ and $a^4(x^2)\neq x^2$, we get $a^{4\cdot 3^n}(x^{3^n})=x^{3^n}$ and $a^{4\cdot 3^n}(x^{2\cdot 3^n})\neq x^{2\cdot 3^n}$ for all $n\in\mathbb{N}$. It follows that the size of the orbit of $x^m$ under $a$ is not bounded by an absolute constant independently on $m$. Thus $a$ has infinite order and $G_A$ is infinite. The claim is proved.

Another implication is that $a$ and $x$ viewed as elements of $\pi_1(\SqComplex_S)$ generate an anti-torus.
Indeed, let us assume that $a^{n_0}x^{m_0}=x^{m_0}a^{n_0}$ in $\pi_1(\SqComplex_S)$ for some $n_0,m_0\in\mathbb{N}$. Then $a^{n_0}(x^m)=x^m$ for all $m\in\mathbb{N}$, which implies that the orbit of $x^m$ under $a$ contains at most $n_0$ elements. We get a contradiction. The same arguments work for $\langle a,y\rangle$, $\langle b,x\rangle$, and $\langle b,y\rangle$.
\end{proof}

The complex $\SqComplex_S$ is the only complete $\VH$ complex with four squares for which we do not know
whether its fundamental group is residually-finite. Our computations support the following conjecture
formulated in \cite[Conjecture~23]{Rattaggi:ExaSQ}:

\begin{conjecture}
The group $\pi_1(\SqComplex_S)$ is residually finite and just-infinite.
\end{conjecture}

If this conjecture is true, then $\SqComplex_S$ would be the smallest example of an irreducible complete $\VH$ square complex having a residually finite fundamental group, while $\SqComplex_D$ would be
the only complete $\VH$ square complex with four $2$-cells having a non-residually finite fundamental group. Actually, the fact proved in Theorem~\ref{thm_SqComplex_S} that $G_A$ and $G_{\dual A}$ are free groups strongly suggest that $\pi_1(\SqComplex_S)$ is just-infinite (see Remark~\ref{rem_just_infinite_free}). We have checked with GAP that there is no self-inverse self-dual automaton with less than four states generating a free group. Therefore, the automaton in Figure~\ref{fig_complex_BM44_S4} is a smallest automaton with such properties.

\section{Finitely presented torsion-free simple groups}\label{sect_FPSimple}

In \cite{BM:FP_simple,BM:LatticesProductTrees} Burger and Mozes constructed the first examples of finitely presented torsion-free simple groups. Their construction relies on two sufficient conditions for the fundamental groups of complete $\VH$ complexes with one vertex: one guarantees non-residual finiteness \cite[Proposition~2.1]{BM:LatticesProductTrees}, while the other one implies just-infiniteness \cite[Theorem~4.1]{BM:LatticesProductTrees}. The non-residually finite groups coming from the first condition always have a nontrivial normal subgroup of infinite index. However, one can embed a group satisfying the first condition into a group satisfying the second condition and construct virtually simple groups. The smallest simple group constructed in this way in \cite{BM:FP_simple,BM:LatticesProductTrees} has amalgam decomposition of the form $F_{7919}*_{F_{380065}}F_{7919}$ and $F_{47}*_{F_{364321}}F_{47}$.

In \cite{Rattaggi:PhD} Rattaggi followed this approach, but started with the non-residually finite group $\pi_1(\SqComplex_W)*_{\langle a,b,c\rangle}\pi_1(\SqComplex_W)$ constructed by Wise. In this way Rattaggi constructed a finitely presented torsion-free simple group that possesses amalgam decompositions $F_7*_{F_{73}}F_7$ and $F_{11}*_{F_{81}}F_{11}$. We get somewhat ``smaller'' simple groups by starting with the non-residually finite groups $\pi_1(\SqComplex_A)$ and $\pi_1(\SqComplex_B)$ for the Aleshin and Bellaterra automata.

\begin{example}\label{thm_Simple_Aleshin}
The group $\Gamma=\langle a_1,a_2,a_3,a_4,b_1,b_2,b_3,b_4 |  R_{4,4}\rangle$, where
\[
R_{4,4}=\left\{
\begin{aligned}
\underline{a_1b_1}&\underline{=b_1a_2}, & \underline{a_1b_2}&\underline{=b_2a_2}, & a_1b_3&=b_3a_1, & a_1b_4&=b_4^{-1}a_1,\\
\underline{a_2b_1}&\underline{=b_2a_1}, & \underline{a_2b_2}&\underline{=b_1a_3}, & a_2b_3&=b_3a_2, & a_2b_4&=b_4^{-1}a_2, \\
\underline{a_3b_1}&\underline{=b_2a_3}, & \underline{a_3b_2}&\underline{=b_1a_1}, & a_3b_3&=b_4a_3, & a_3b_4&=b_3a_4,\\
a_4b_1&=b_3a_4^{-1}, & a_4b_2&=b_4^{-1}a_4^{-1}, & a_4b_3&=b_4a_3^{-1}, & a_4b_2^{-1}&=b_1a_4^{-1}
\end{aligned}
\right\},
\]
is non-residually finite and just-infinite. The intersection $\Gamma_0$ of subgroups of finite index is a subgroup of
index $4$. The group $\Gamma_0$ is a finitely presented torsion-free simple group, and it decomposes into the amalgamated product $F_7*_{F_{49}}F_7$.
\end{example}
\begin{proof}
The first six relations indicate that $\Gamma$ contains $\pi_1(\SqComplex_A)$ for the Aleshin automaton. Hence, $\Gamma$ is non-residually finite and $\Gamma_0$ contains $(b_1^{-1}b_2)^4$. One can directly check that $\Gamma$ satisfies the other conditions of \cite[Corollary~5.4]{BM:LatticesProductTrees} (here $P_h\cong P_v\cong A_8$), which implies that $\Gamma_0$ is simple. One can directly check with GAP that adding the relation $(b_1^{-1}b_2)^4$ to the presentation of $\Gamma$ leads to a finite group of order $4$. It follows that $\Gamma_0$ is the normal closure of $(b_1^{-1}b_2)^4$ and has index $4$ in $\Gamma$. The amalgam decomposition of $\Gamma_0$ follows from \cite[Proposition~1.4]{Rattaggi:PhD}.
\end{proof}

The Bellaterra automaton provides a simple group with a slightly smaller presentation.

\begin{example}\label{thm_Simple_Bellaterra}
The group $\Gamma=\langle a_1,a_2,a_3,a_4,b_1,b_2,b_3,b_4 |  R_{4,4}\rangle$, where
\[
R_{4,4}=\left\{
\begin{aligned}
\underline{a_1b_1}&\underline{=b_2a_2}, & \underline{a_1b_2}&\underline{=b_1a_2}, & a_1b_3&=b_3a_1, & a_1b_4&=b_4^{-1}a_1,\\
\underline{a_2b_1}&\underline{=b_1a_1}, & \underline{a_2b_2}&\underline{=b_2a_3}, & a_2b_3&=b_3a_2, & a_2b_4&=b_4^{-1}a_2,\\
\underline{a_3b_1}&\underline{=b_1a_3}, & \underline{a_3b_2}&\underline{=b_2a_1}, & a_3b_3&=b_4a_4, & a_3b_4&=b_3a_3^{-1},\\
a_4b_1&=b_1a_4, & a_4b_2&=b_3a_4, & a_4b_3&=b_4a_3^{-1}, & a_4b_4&=b_2a_4
\end{aligned}
\right\},
\]
is non-residually finite and just-infinite. The intersection $\Gamma_0$ of subgroups of finite index is a subgroup of
index $4$. The group $\Gamma_0$ is a finitely presented torsion-free simple group, and it decomposes into the amalgamated product $F_7*_{F_{49}}F_7$. The group $\Gamma_0$ has a finite presentation with $23$ generators and $56$ relations of total length $216$ given in Table~\ref{tab:PresentationSimple}.
\end{example}
\begin{proof}
The proof is as above. The group $\Gamma$ contains $\pi_1(\SqComplex_B)$ for the Bellaterra automata and satisfies \cite[Corollary~5.4]{BM:LatticesProductTrees}. It follows that $\Gamma_0$ is simple and contains $a_1^2a_2^{-2}$ (see Example~\ref{ex_Bellaterra}). Adding the relation $a_1^2a_2^{-2}$ to the presentation of $\Gamma$ leads to a finite group of order $4$. Hence, $\Gamma_0$ is the normal closure of $a_1^2a_2^{-2}$ and has index $4$ in $\Gamma$.
\end{proof}

\begin{table}
\[
\left\{\begin{matrix}
s_{13}s_{8}^{-1}s_{10}, & s_{23}s_{2}s_{22}, & s_{13}^{-1}s_{1}s_{12}, & s_{21}s_{2}^{-1}s_{14}^{-1} \\
s_{13}^{-1}s_{12}s_{1}, & s_{11}s_{4}s_{9}^{-1}, & s_{20}s_{2}s_{16}^{-1}, & s_{17}s_{2}^{-1}s_{15}^{-1} \\
s_{10}s_{6}^{-1}s_{8}^{-1}s_{6}, & s_{11}s_{2}s_{10}^{-1}s_{2}^{-1}, & s_{11}s_{3}s_{10}^{-1}s_{3}^{-1}, & s_{12}s_{11}^{-1}s_{5}s_{9} \\
s_{12}s_{8}^{-1}s_{1}s_{10}, & s_{9}s_{2}^{-1}s_{8}^{-1}s_{2}, & s_{14}s_{7}s_{9}^{-1}s_{7}^{-1}, & s_{10}s_{7}^{-1}s_{8}^{-1}s_{7} \\
s_{16}s_{7}^{-1}s_{15}^{-1}s_{4}, & s_{10}s_{4}^{-1}s_{8}^{-1}s_{4}, & s_{16}s_{5}^{-1}s_{15}^{-1}s_{7}, & s_{16}s_{4}^{-1}s_{15}^{-1}s_{5} \\
s_{17}s_{7}^{-1}s_{11}^{-1}s_{7}, & s_{11}s_{5}s_{9}^{-1}s_{1}^{-1}, & s_{17}s_{5}^{-1}s_{14}^{-1}s_{4}, & s_{17}s_{4}^{-1}s_{14}^{-1}s_{6} \\
s_{17}s_{3}^{-1}s_{15}^{-1}s_{1}, & s_{10}s_{5}^{-1}s_{8}^{-1}s_{5}, & s_{9}s_{3}^{-1}s_{8}^{-1}s_{3}, & s_{17}s_{13}s_{15}^{-1}s_{2} \\
s_{15}s_{6}s_{9}^{-1}s_{6}^{-1}, & s_{21}s_{3}^{-1}s_{14}^{-1}s_{1}, & s_{16}s_{6}^{-1}s_{11}^{-1}s_{6}, & s_{13}^{-1}s_{4}^{-1}s_{13}^{-1}s_{2} \\
s_{5}^{-1}s_{13}s_{3}s_{13}, & s_{4}^{-1}s_{13}s_{2}s_{13}, & s_{17}s_{6}^{-1}s_{14}^{-1}s_{5}, & s_{22}s_{5}^{-1}s_{20}s_{4} \\
s_{2}s_{16}s_{13}s_{20}^{-1}, & s_{3}s_{16}s_{12}s_{20}^{-1}, & s_{13}s_{11}^{-1}s_{4}s_{9}, & s_{17}s_{12}s_{15}^{-1}s_{3} \\
s_{20}s_{7}s_{13}s_{7}^{-1}, & s_{21}s_{6}^{-1}s_{13}^{-1}s_{6}, & s_{23}s_{5}s_{21}^{-1}s_{4}^{-1}, & s_{13}^{-1}s_{5}^{-1}s_{13}^{-1}s_{3} \\
s_{21}s_{12}s_{14}^{-1}s_{3}, & s_{21}s_{13}s_{14}^{-1}s_{2}, & s_{22}s_{7}^{-1}s_{13}s_{7}, & s_{22}s_{6}^{-1}s_{20}s_{5} \\
s_{20}s_{3}s_{16}^{-1}s_{1}^{-1}, & s_{22}s_{4}^{-1}s_{20}s_{6}, & s_{23}s_{13}^{-1}s_{22}s_{2}^{-1}, & s_{23}s_{12}^{-1}s_{22}s_{3}^{-1} \\
s_{23}s_{6}s_{13}s_{6}^{-1}, & s_{23}s_{3}s_{22}s_{1}^{-1}, & s_{23}s_{4}s_{21}^{-1}s_{7}^{-1}, & s_{23}s_{7}s_{21}^{-1}s_{5}^{-1}
\end{matrix}\right\}
\]
\caption{\label{tab:PresentationSimple}Relators of the torsion-free simple group $\Gamma_0$ from Example~\ref{thm_Simple_Bellaterra}}
\end{table}

\end{document}